\documentclass[12pt]{article}
\usepackage{bbm}
\usepackage{graphicx}
\graphicspath{ {./images/} }
\usepackage{wrapfig}
\usepackage[mathscr]{euscript}
\usepackage{subcaption}
\usepackage{hyperref}
\usepackage{epstopdf}
\usepackage{enumerate}
\usepackage{amsmath,amsfonts,amssymb,amsthm,epsfig,epstopdf,titling,url,array}
\usepackage{xcolor}
\usepackage{framed}
\usepackage[utf8]{inputenc}
\usepackage[english]{babel}
\usepackage{xparse}
\usepackage{authblk}
\usepackage{tikz}
\usepackage{placeins}
\usepackage{float}
\usetikzlibrary{shapes.geometric}
\usepackage[margin=2.4cm]{geometry}
\usepackage{authblk}
\usepackage{booktabs} % For professional table lines
\newtheorem{lem}{Lemma}[section]
\newtheorem{prop}{Proposition}[section]

\newtheorem{rem}{Remark}[section]

\newtheorem{maintheorem}{Theorem}

\date{}
\begin{document}
	\title{Chebyshev's method applied to polynomials with rotational symmetry}
	%\iffalse
	%-------adding authors--------------
	\author[1]{Tarakanta Nayak 
		\footnote{tnayak@iitbbs.ac.in}}
	\author[1]{Pooja Phogat  \footnote{Corresponding author, poojaphogat174acad@gmail.com}}
	\affil[1]{Department of Mathematics, 
		Indian Institute of Technology Bhubaneswar, India}
	\date{}
	%------------------------------------
	%\fi
	\maketitle
	%	\begin{verbatim}\end{verbatim}\vspace{2.5cm}
	\begin{abstract}
We investigate the dynamics of Chebyshev's method applied to the polynomial family $p_n(z)=z(z^n-1)$ for $n>1$. The resulting map is denoted by  $C_n$. It is proved that the immediate basins corresponding to the non-zero roots are unbounded and simply connected. We also show that the Julia set of $C_n$ is connected. It is proved that the immediate basin of the root at the origin exhibits a different behavior: it is unbounded for $n\leq 16$ and bounded for $n\geq 17$. We establish that $C_n$ is convergent whenever $n\leq 16$ or $n$ is odd. Finally, we determine the symmetry group of $C_n$ and prove that it coincides with the symmetry group of the polynomial $p_n$, thereby confirming, for this family, a conjecture proposed by Nayak and Pal.
	\end{abstract}
	
	\section{Introduction}
	The problem of approximating the roots of nonlinear equations is one of the oldest and most fundamental problems in mathematics. For a polynomial $p$, a root-finding method assigns to $p$ a rational map $F_p$ such that every root $z_0$ of $p$ is an attracting fixed point of $F_p$, i.e., $F_p (z_0)=z_0$ and  $|F_p '(z_0)|<1$. Starting from an initial guess $\alpha_0$, the method generates a sequence $\alpha_{n+1}=F_p(\alpha_n)$ with the aim of converging to a root of $p$. The set of all initial points converging to a given root is called its basin of attraction.
	
	Beginning with Newton's method, numerous iterative schemes, including those of Halley, Traub, Chebyshev, Householder, and Ostrowski, have been developed to improve convergence and computational efficiency. Beyond their numerical significance, these methods naturally give rise to dynamical systems involving  rational maps on the Riemann sphere $\widehat{\mathbb C}$, establishing a connection between numerical analysis and complex dynamics. The iterative behavior of these maps is studied through the associated Fatou and Julia sets. 
	
	The Fatou set of a rational map $R$, denoted by $\mathcal{F}(R)$, is the maximal open subset of $\widehat{\mathbb{C}}$ on which the family $\{R^n:n\in\mathbb{N}\}$ is equicontinuous. The Julia set, denoted by  $\mathcal{J}(R)$ is the complement of the Fatou set in $\widehat{\mathbb{C}}$. A maximal connected component of the Fatou set is called a Fatou component.
	In the context of root-finding methods, Fatou components often coincide with attracting basins of the roots, whereas the Julia set forms the boundary of these basins. Necessary background is provided in Section~\ref{S2_prelim}.
	
	The dynamical study of root-finding methods has revealed a close relationship between the algebraic properties of a polynomial and the topology of the associated Fatou components. Motivated by these connections, we investigate the dynamics of Chebyshev's method, with particular emphasis on the topology of immediate basins of attraction and the influence of polynomial symmetries on the associated dynamical plane.
	
	The dynamics of Chebyshev's method have attracted considerable attention in recent years, primarily because of the existence of non-repelling extraneous fixed points and its third order convergence. Garc\'{i}a-Olivo et al. \cite{GGM2015} initiated its dynamical study by proving the existence of attracting extraneous fixed points for cubic polynomials. Campos, Canela, and Vindel \cite{CCV2020} investigated the Chebyshev--Halley family applied to the polynomials $z^n+c$, obtaining criteria for the simple connectivity of Fatou components and the connectedness of Julia sets. Guti\'{e}rrez and Varona \cite{GV2020} established the existence of superattracting extraneous fixed points and periodic cycles for cubic polynomials and analysed their influence on the parameter space.
	
	Nayak and Pal \cite{Nayak-Pal2022} carried out a systematic study of the  Chebyshev's method for cubic polynomials, determining the possible degrees of the associated rational maps, proving connectedness of the Julia sets, and parametrising the generic family by the multiplier of an extraneous fixed point. More recently, Nayak and Pal \cite{Sym_dyn} investigated the symmetry groups of Chebyshev maps, relating them to the symmetries of the underlying polynomial, and establishing the locally connectedness of Julia sets. Furthermore, Nayak and Phogat \cite{Nayak-Phogat2025} studied Chebyshev's method for polynomials with exactly two distinct roots, proving that the Julia set is connected and that the Fatou set consists precisely of the attracting basins of the roots. 
	
For $n>1$, we consider the Chebyshev map $C_n$ associated to the polynomial $p_n(z) =z(z^n-1)$. The motivation for considering $p_n$ is not only that the Halley's method applied to it has already been investigated \cite{CGJ2025,halley-liu-etal-2025} but also because it is the simplest polynomial with rotational symmetry. Furthermore, Chebyshev's method applied to $A z(z^n+a)$ for any $ A,a\in\mathbb{C}\setminus\{0\}$ reduces to that of $p_n$. This is a consequence of the Scaling property and is discussed in Subsection~\ref{dyn_results}.
	\par
The basin of attraction of a (super)attracting fixed point $z_0$ of $C_n$ is the set $\{z \in \widehat{\mathbb{C}}: \lim_{m \to \infty} C_n ^m (z)=z_0\}$. The connected component of the basin containing $z_0$ is called the immediate basin of $z_0$. A central problem in the dynamics of root-finding methods is to understand the topology of the immediate basins of attraction and their relation to the geometry of the Julia set. In particular, simple connectivity and unboundedness of immediate basins often lead to information about the global dynamics. Our first two results establish these properties for the non-zero roots of $p_n$.
	
	\begin{maintheorem}\label{Imm_unbd}
		The immediate basins of $C_n$ corresponding to every non-zero root of $p_n$ are unbounded and simply connected for all $n$.
	\end{maintheorem}
	
	\begin{maintheorem}\label{Connected_J.set}
		The Julia set of $C_n$ is connected  for all $n$.
	\end{maintheorem}
 For the Newton's method applied to a polynomial $p$, all the immediate basins corresponding to the roots of $p$ are unbounded (see \cite[Proposition 6]{HSS2001}). In contrast, for the family $p_n(z)=z(z^n-1)$, $ n>1$, bounded immediate basins were recently shown to occur for Halley's method \cite{CGJ2025,halley-liu-etal-2025}. We show that Chebyshev's method exhibits a similar behavior.
 
\begin{maintheorem}\label{Bdd_imm}
	The immediate basin $\mathcal{A}_0$ of $C_n$ corresponding to the origin is
	\begin{enumerate}
		\item unbounded whenever $n\leq 16$,
		\item bounded whenever $n\geq 17$.
	\end{enumerate}
\end{maintheorem}	

For a polynomial $p$, a root-finding method $F_p$ is said to be \emph{convergent} if the Fatou set $\mathcal{F}(F_p)$ consists entirely of the basins of attraction of the roots of $p$ (see \cite{McMullen1987}). Convergence is a fundamental requirement for the practical effectiveness of an iterative method, as it guarantees that \emph{almost every} initial guess eventually converges to a root.

\begin{maintheorem}\label{Convergent}
	If $n\leq 16$ or $n$ is odd, then $C_n$ is convergent.
\end{maintheorem}

The Julia set of a rational function is often preserved by simple maps, and studying these maps can simplify its dynamics. The collection of all holomorphic Euclidean isometries $\sigma(z)=az+b$, $|a|=1$, satisfying $\sigma(\mathcal{J}(R))=\mathcal{J}(R)$ forms the \emph{symmetry group} of $R$, denoted by $\Sigma R$.

This group is well-classified for polynomials (see \cite[Theorem 9.5.4]{Beardon_book}). Consider a monic polynomial $p$ whose second leading coefficient is zero. Such a polynomial is called \emph{normalized}, and can be written as $p(z)=z^k P_0(z^m)$, where $P_0$ is a monic polynomial, and $k, m$ are non-negative integers, maximal for the expression. Then $\Sigma p=\{z\mapsto \lambda z: \lambda^m =1\}.$   

The symmetry group of a polynomial $p$ and that of a root-finding method applied to $p$ are related. In \cite{Sym_dyn}, it is proved that whenever a root-finding method $F_p$ satisfies the Scaling property, then for a normalized polynomial $p$, we have $\Sigma p\subseteq \Sigma F_p$. Since Chebyshev's method satisfies this property (see Lemma~\ref{Scaling}), it was conjectured that $\Sigma p=\Sigma C_p$. We confirm this conjecture for the family under consideration.
\begin{maintheorem}\label{Equal_sym}
	The symmetry groups of $p_n$ and $C_n$ are identical.
\end{maintheorem}

The article is organized as follows. In Section~\ref{S2_prelim}, we recall several preliminary  results and facts that will be used throughout the paper. In Section~\ref{S3_Fun-prop}, we investigate various functional properties of the Chebyshev map $C_n$. Section~\ref{S4_Dynamics} is devoted to the dynamical study of $C_n$, where we prove our main results. Finally, in Section~\ref{S5_Con-rem}, we compare the dynamics of Newton's, Halley's, and Chebyshev's methods and present some concluding remarks.

\section{Background and preliminaries}\label{S2_prelim}
For completeness and future reference, we summarize in this section several standard results that will be frequently used. The first subsection concerns with fixed points and the Scaling property of Chebyshev's method while the second focuses on  general dynamical properties of rational functions.
\subsection{Fixed points and the Scaling property of Chebyshev's method}
Chebyshev's method $C_p$ applied to a polynomial $p$ is given by
\begin{equation}\label{Formula-C_p}
	C_p(z)=z-\left(1+\frac{1}{2}L_p(z)\right)\frac{p(z)}{p'(z)},
\end{equation}
where $L_p(z)=\frac{p(z)p''(z)}{(p'(z))^2}.$
This is a third-order convergent method, i.e., the local degree of $C_p$ at every root of $p$ is at least $3$. 
\par 

A point $z_0 \in \widehat{\mathbb{C}}$ is a \textit{fixed point} of a rational map $R$ if $R(z_0)=z_0$.  
The \textit{multiplier} $\lambda_{z_0}$ of a fixed point $z_0$ is defined as $R'(z_0)$ if $z_0$ is finite and as $S'(0)$ if $z_0=\infty$, where $S(z)=\tfrac{1}{R(1/z)}$. A fixed point $z_0$ is called attracting, neutral, or repelling if $|\lambda_{z_0}|<1, =1$, or $>1$, respectively.
An attracting fixed point is called \textit{superattracting} if its multiplier is $0$. 

A (super)attracting fixed point $z_0$ of $R$ is always in the Fatou set. Moreover, there exist Fatou components consisting of points whose forward orbits converge to $z_0$. The union of all such components is actually the basin of attraction of $z_0$ and is denoted by $\mathcal{B}_{z_0}$ (or $\mathcal{B}_{R,z_0}$, when necessary). The Fatou component of $\mathcal{B}_{z_0}$ that contains $z_0$ is called the immediate basin of $z_0$ and is denoted by $\mathcal{A}_{z_0}$ (or $\mathcal{A}_{R,z_0}$, when necessary).  If $z^*$ is a parabolic fixed point i.e., $R'(z_0)=e^{2\pi i \theta}$, where $\theta$ is rational, then it is in the Julia set. However, like a (super)attracting fixed point, it has a basin containing points whose successive iterations eventually converge to $z^*$. We denote the parabolic basin with the same notation $\mathcal{B}_{z^*}$. In this case, the immediate basin  $\mathcal{A}_{z^*}$ is defined as the union of components of $\mathcal{B}_{z^*}$ whose boundary contains $z^*$.
\par 
A point $w_0 \in \widehat{\mathbb{C}}$ is called a $k$-periodic point of $R$ if $k$ is the smallest natural number such that $R^k (w_0)=w_0$. It is called attracting, neutral, or repelling if $w_0$ is attracting, neutral or repelling, respectively as a fixed point of $R^k$.
\par
It follows from Equation~\eqref{Formula-C_p} that, in contrast to Newton's method, $C_p$ possesses finite extraneous fixed points, namely the solutions of $L_p(z)=-2$. The fixed points of $C_p$ are classified as follows.

\begin{lem}(\cite[Proposition 2.3]{Nayak-Pal2022})
	A root of a polynomial $p$ with multiplicity $k$ is an attracting fixed point of $C_p$ with multiplier $\frac{(k-1)(2k-1)}{2k^2}$. The point at $\infty$ is a repelling fixed point of $C_p$ with multiplier $\frac{2d^2}{2d^2 -3d+1}$, where $d$ is the degree of $p$. If $e_0$ is an extraneous fixed point, its multiplier is given by $2\left(3-L_{p'}(e_0)\right)$.
	\label{multiplier-fixedpoints}
\end{lem}
It is important to note that an extraneous fixed point of Chebyshev's method may be non-repelling. In \cite{Nayak-Pal2022}, the family of cubic polynomials is parameterized by the multiplier of the extraneous fixed point arising from the associated Chebyshev's method. 

Chebyshev's method satisfies the Scaling property stated in the next lemma.
\begin{lem}\cite[Theorem~2.2]{Nayak-Pal2022}
Let $p$ be a polynomial of degree at least two. If $T(z) = \alpha z + \beta$ with $\alpha, \beta \in \mathbb{C}, \alpha \neq 0$, and $g = \lambda p \circ T$ for $\lambda \neq 0$, then 
	$T \circ C_g \circ T^{-1} = C_p.$ 
	\label{Scaling}
\end{lem}
The Fatou and the Julia sets of $C_p$ and $C_g$ are \textit{essentially} the same.
Choosing suitable values of  $\alpha, \beta, \lambda$, it is possible to have $g$ in a  form much simpler than $p$. Then it becomes enough to study Chebyshev's method applied to $g$ in order to understand that of $p$. In order to understand Chebyshev's method applied to a polynomial $p (z)=A z(z^n+a)$ for any $ A,a\in\mathbb{C}\setminus\{0\}$, it is enough to study $C_n$. This can be seen by taking $\alpha$ such that $\alpha^n =-a$ and $\lambda =\frac{1}{A \alpha^{n+1}}$ in Lemma~\ref{Scaling}.

Another interesting situation arises in the context of symmetry groups. 
For an arbitrary polynomial $p$, it is possible to choose $\alpha, \beta, \lambda$ such that $\lambda p \circ T$ is \emph{normalized}. As mentioned in the introduction, the symmetry groups of a polynomial and its Chebyshev's method are related. A more precise statement is the following. 
\begin{lem}(\cite[Corollary 1.1.1]{Sym_dyn})\label{Sym_C_p}
Let $p$ be a normalized polynomial. Then $\Sigma p \subseteq \Sigma C_p.$
\end{lem}
It follows from \cite[Theorem~1.2]{Sym_dyn} that $\Sigma C_p$ does not contain any translation. Therefore, lemma~\ref{Sym_C_p} asserts that if $\Sigma p$ is non-trivial (i.e., if it contains at least one non-identity map) then $\Sigma C_p$ contains only rotations about the origin.	

\subsection{Some results on the dynamics of rational functions}\label{dyn_results}
	
A Fatou component $U$ of $R$ is called $k$-periodic if $k$ is the smallest natural number such that $R^k (U)=U$.

For a rational function $R$, every Fatou component $U$ is either periodic or pre-periodic (i.e., $R^n(U)$ is periodic for some $  n $). This is the famous No-Wandering Domain Theorem by Sullivan (see \cite[Theorem~8.1.2]{Beardon_book}). Furthermore, there are four types of periodic Fatou components of $R$.
\begin{itemize}
\item \textit{Attracting domain:} $U$ is called an attracting domain if it contains a $k$-periodic point $z_{0}$ such that $R^{nk}(z)\rightarrow z_{0}\ \mbox{as}\ n\rightarrow \infty \mbox{~for all}\ z\in U.$ If $z_0$ is superattracting then $U$ is called a superattracting domain.
\item \textit{Parabolic domain:} $U$ is said to be a parabolic domain if the boundary $\partial U$ contains a parabolic $k$-periodic point $z_{0}$ such that $R^{nk}(z)\rightarrow z_{0}$ as $n\rightarrow \infty$ for all $z\in U$. 
\item \textit{Siegel disk:} $U$ is a Siegel disk if there exists an analytic homeomorphism $\phi: U\rightarrow \{z:|z|<1\}$ such that $\phi \circ R^{k} \circ\phi^{-1}(z)=e^{2 \pi i \alpha} z$ for some irrational number $\alpha$. 
\item \textit{Herman ring:} $U$ is a Herman ring if there exists an analytic homeomorphism $\phi:U\rightarrow \{z:1<|z|<r\}$ such that $\phi \circ R^{k} \circ\phi^{-1}(z)=e^{2 \pi i \alpha} z$ for some irrational number $\alpha$.
	\end{itemize}
	
	Note that the   attracting and parabolic domains can be  simply or infinitely connected, whereas a Siegel disk is always simply connected and a Herman ring is always doubly connected.
	We need the following three well-known results in complex dynamics. The first result shows the relation between periodic Fatou components and critical points of $R$.  A critical point of $R$ is a point where the local degree of $R$ is at least two. For a finite critical point $c$, either $R'(c)=0$ or $c$ is a multiple pole of $R$.
 	
	\begin{lem} (\cite{Beardon_book}) \label{cpoint}
		Let $U$ be a periodic Fatou component of a rational map $R$.  
		If $U$ is an immediate attracting or parabolic basin, then $U$ contains at least one critical point of $R$. If $U$ is a Siegel disk or a Herman ring, then the boundary of $U$ is contained in the closure of the post-critical set   $\{R^n (c):   ~c~\mbox{is a critical point of}~R: n \geq 0\}$.
	\end{lem}
	There is an interesting result relating the fixed points and the connectivity of the Julia set. 	A fixed point is called weakly repelling if it is either repelling or has multiplier $1$. Shishikura proved that if the Julia set of a rational map $R$ is disconnected, then $R$ has at least two weakly repelling fixed points (see \cite[Theorem~1]{Shishikura2009}). Through this result, the connectedness of the Julia set of Newton's method applied to polynomials is established.
	
	One of the main goals of the present work is to prove the connectedness of the Julia set of Chebyshev's method. By analysing the location of the poles, \cite[Lemma~4.3]{Nayak-Pal2022} and \cite[Lemma~3.5]{Sym_dyn} provide some sufficient conditions for the connectedness of the Julia set. In the following lemma, we present these two results together. By a Julia component, we mean a maximally connected subset of the Julia set.
	\begin{lem}(\cite[Lemma~4.3]{Nayak-Pal2022}, \cite[Lemma~3.5]{Sym_dyn}) \label{Jconnected}
		Let $R$ be a rational function for which $\infty$ is a repelling fixed point. Then the boundary of an unbounded immediate basin of attraction contains at least one pole. Furthermore, if all poles of $R$ lie on the unbounded Julia component, then the Julia set of $R$ is connected.
	\end{lem}
	The following lemma provides a useful criterion for the boundedness of the pre-images of an unbounded immediate attracting basin whenever the Julia set is locally connected.
	\begin{lem}(\cite[Lemma~3.3]{Sym_dyn})\label{bdd_pre-im}
		Let $R$ be a rational function with locally connected Julia set. Also assume that $\infty$ is a repelling fixed point of $R$. If $\mathcal{A}$ is an unbounded immediate basin corresponding to an attracting fixed point of $R$ then all pre-images of $\mathcal{A}$ are bounded.
	\end{lem}

	\section{Basic properties of  $C_n$}\label{S3_Fun-prop}
Recall that $$p_n(z)=z(z^n-1).$$
Throughout this article, we assume that $n>1$. An immediate and useful consequence of this assumption is that  $\Sigma p_n$ is non-trivial. The corresponding Chebyshev map is given by
	% Then  $p'(z)=(n+1)z^n-1, p''(z)=n(n+1)z^{n-1}~\mbox{and}~ p'''(z)=n(n+1)(n-1)z^{n-1}$. Consequently, 
	%Note that
	%$$	L_p(z)   =\frac{n(n+1)z^n(z^n-1)}{\left((n+1)z^n-1\right)^2},$$
%	and 
	\begin{align}
		C_n(z) &  = z-z(z^n-1)\left(\frac{(n+1)(3n+2) z^{2n}-(n+1)(n+4) z^n+2}
		{2\big((n+1)z^n-1\big)^3}\right) \label{eq:Cp}\\
		& =\frac{z^{n+1}\big(n(n+1)(2n+1)z^{2n}-2n(n+1)z^n-n(n-1)\big)}
		{2\big((n+1)z^n-1\big)^3}.  \label{eq:Cp2}
	\end{align}
Note that $C_n$ is an odd function when $n$ is even, whereas for odd $n$, $C_n$ is neither even nor odd.
\subsection{Extraneous fixed points}
 It follows from Equation~\eqref{eq:Cp} that the extraneous fixed points of $C_n$ are precisely the roots of  $(n+1)(3n+2)z^{2n}-(n+1)(n+4)z^n+2=0.$ These are given explicitly by one of the following equations
	\begin{equation}\label{ext_f1}
		z^n=\frac{(n+1)(n+4)- n \sqrt{(n+1)(n+9)}}{2(n+1)(3n+2)},
	\end{equation}
	\begin{equation}\label{ext_f2}
		z^n=\frac{(n+1)(n+4)+ n \sqrt{(n+1)(n+9)}}{2(n+1)(3n+2)}.
	\end{equation}
Thus, the number of finite extraneous fixed points of $C_n$ is $2n$. Moreover, as $(n+1)(n+4) - n \sqrt{(n+1)(n+9)} > 0$ for all $n > 1,$ each of  Equations~\eqref{ext_f1} and~\eqref{ext_f2} yields exactly one positive extraneous fixed point, and these are distinct. The following is immediate.
%	 Let $e_1$ and $e_2$ denote the smaller and larger of these positive extraneous fixed points, respectively. If $n$ is even, then $-e_1$ and $-e_2$ are also extraneous fixed points, whereas if $n$ is odd, then there are no other real extraneous fixed points. These observations are summarized in the following lemma.

\begin{lem}
	For each $n$, the map $C_n$ has $2n$ finite extraneous fixed points, exactly two of which are real and positive, and we denote these by $e_1$ and $e_2$ with $e_1 < e_2$. Moreover, the following hold.
	\begin{enumerate}
		\item If $n$ is even, then $-e_1$ and $-e_2$ are the only other real extraneous fixed points of $C_n$.
		\item If $n$ is odd, then there is no other real extraneous fixed point.
	\end{enumerate}
	\label{extraneous-general}
\end{lem}

%
% From Lemma~\ref{extraneous-general}, it follows that
%	$$e_1^n = \frac{(n+1)(n+4) - n \sqrt{(n+1)(n+9)}}{2(n+1)(3n+2)} \quad \text{and} \quad
%	e_2^n = \frac{(n+1)(n+4) + n \sqrt{(n+1)(n+9)}}{2(n+1)(3n+2)}.$$
Equation~\eqref{eq:Cp2} shows that $C_n$ has a unique positive pole, namely the positive root of $z^n =\frac{1}{n+1}$. Denoting this pole by $\xi$, we now determine its location relative to the positive extraneous fixed points. 
\begin{lem}
Let  $e_1$ and $e_2$ denote the two positive extraneous fixed points, and $\xi$ be the unique positive pole of $C_n$. Then $0<e_1<\xi<e_2<1$.
		\label{extraneous}
	\end{lem}
\begin{proof}
Since $e_1>0$, it is enough to show that $e_2<1$ and $e_1< \xi <e_2$ for all $n> 1$.
		
To show $e_2<1$, it suffices to prove that $e_2^n<1$, that is, 
$$\frac{(n+1)(n+4)+ n \sqrt{(n+1)(n+9)}}{2(n+1)(3n+2)}<1.$$
Equivalently, $(n+1)(n+4)+ n \sqrt{(n+1)(n+9)}<2(n+1)(3n+2),$ i.e., 
$\sqrt{(n+1)(n+9)}<5(n+1).$ This simplifies to $24n+16>0$, which clearly holds for all $n>1$. 
		
		Since $\xi^n=\frac{1}{n+1}$, in order to prove $e_1<\xi<e_2$, it is enough to show that $e_1^n < \xi ^n< e_2^n$, i.e., 
		$$\frac{(n+1)(n+4)- n \sqrt{(n+1)(n+9)}}{2(n+1)(3n+2)}<\frac{1}{n+1}<\frac{(n+1)(n+4)+ n \sqrt{(n+1)(n+9)}}{2(n+1)(3n+2)}.$$
		Multiplying through by the positive quantity $2(n+1)(3n+2)$ gives 
		$$(n+1)(n+4)- n \sqrt{(n+1)(n+9)}<2(3n+2)<(n+1)(n+4)+ n \sqrt{(n+1)(n+9)}.$$ 
     This is equivalent to 
	 $ -n \sqrt{(n+1)(n+9)}<-n(n-1)< n \sqrt{(n+1)(n+9)} $, and equivalently, 
	 $ - \sqrt{(n+1)(n+9)}<-(n-1)< \sqrt{(n+1)(n+9)}.$  
	This clearly holds for all $n>1$. Hence, $e_1<\xi<e_2$.
	\end{proof}
Let $e_0$ be an extraneous fixed point. By Lemma~\ref{multiplier-fixedpoints}, its multiplier is 
	\begin{align}
		C_n'(e_0) =2\big(3-L_{p'}(e_0)\big) 
		 =\frac{2\left( (n+1)(2n+1)e_0^n+(n-1)\right)}{n(n+1)e_0^n}.
		 \label{ext_mult}
	\end{align}	
	We now determine the nature of the extraneous fixed points.
	\begin{lem}
		All extraneous fixed points of $C_n$ are repelling.
	\end{lem}
	
	\begin{proof}
		First, suppose that $e_0$ satisfies Equation~\eqref{ext_f1}.
%		, i.e., 
%		$$e_0^n=\frac{(n+1)(n+4)-n\sqrt{(n+1)(n+9)}}{2(n+1)(3n+2)}.$$
		Then, by Equation~\eqref{ext_mult}, the multiplier of $e_0$ is $$C_n'(e_0)=\frac{2(2n+1)}{n}+\frac{4(n-1)(3n+2)}{n\left((n+1)(n+4)-n\sqrt{(n+1)(n+9)}\right)}.$$
		Rationalizing the denominator of the second term gives
		$$C_n'(e_0)=\frac{2(2n+1)}{n}+\frac{(n-1)\left((n+1)(n+4)+n\sqrt{(n+1)(n+9)}\right)}{2n(n+1)}.$$
		Since $\frac{2(2n+1)}{n}=4+\frac{2}{n}>4,$ and 	$\frac{(n-1)\left((n+1)(n+4)+n\sqrt{(n+1)(n+9)}\right)}{2n(n+1)}> 0$
		for all $n>1$, it follows that $C_n'(e_0)>4.$
		In particular, $C_n'(e_0)>1.$
		
		Now suppose that $\tilde{e}_0$ satisfies Equation~\eqref{ext_f2}.
		Then, its multiplier is
		$$C_n'(\tilde{e}_0)=\frac{2(2n+1)}{n}+\frac{4(n-1)(3n+2)}{n\left((n+1)(n+4)+n\sqrt{(n+1)(n+9)}\right)}.$$
		Moreover, for all $n>1$, we have
		$$\frac{2(2n+1)}{n}=4+\frac{2}{n}>4 ~\mbox{and}~
		 \frac{4(n-1)(3n+2)}{n\left((n+1)(n+4)+n\sqrt{(n+1)(n+9)}\right)}>0.$$
		Consequently, $C_n'(\tilde{e}_0)>4$.
		Therefore, every extraneous fixed point of $C_n$ is repelling.
	\end{proof}
 On the real axis, it follows from Equation~\eqref{eq:Cp} that
	\begin{align}
		C_n(x)-x =-\frac{(n+1)(3n+2)x(x^n-1)(x^n-e_1^n)(x^n-e_2^n)}{2(n+1)^3\big(x^n-\xi^n\big)^3}.
		\label{eq:CpR}
	\end{align}

The next two lemmas follow directly from Equation~\eqref{eq:CpR} (see Figure~\ref{Graph}). 
	
	\begin{lem}
		% Let $ e_1$ and $ e_2$ be the positive extraneous fixed points, and let $\xi$ be the positive pole of $C_n$. Then, for all $n>1$, 
		The following assertions hold for all $n$.
		\begin{enumerate}
			\item If $x \in (0,e_1)\cup(\xi,e_2) \cup(1, \infty)$, then $ C_n(x)<x$.
			\item If $x \in (e_1, \xi) \cup (e_2,1)$, then $ C_n(x)>x$.
			\item Furthermore, $\lim\limits_{x \to {\xi}^-}C_n(x)=+\infty  ~~\text{and}~~\lim\limits_{x \to {\xi}^+}C_n(x)=-\infty.$
		\end{enumerate}
		\label{C_p-real-all}
	\end{lem}   
	\begin{lem}
		The following statements hold for all $n$.
		\begin{enumerate}
			\item If $n$ is odd, then
			$ 
			C_n(x)>x \text{ for all } x\in(-\infty,0).
			$ 
			\item If $n$ is even, then
			 $C_n(x)>x \text{ for all } x\in(-\infty,-1)\cup(-e_2,-\xi)\cup(-e_1,0),$ 
			whereas
			 $C_n(x)<x \text{ for all } x\in(-1,-e_2)\cup(-\xi,-e_1).$ 
		\end{enumerate}
		\label{C_p-real-odd-even}
	\end{lem} 
	
	%%%%%%%%%%%%%%%%%%%%%%%%%%%%%%%%%%%
%	\begin{figure}[h!]
%		\begin{subfigure}{.5\textwidth}
%			\centering
%			\includegraphics[width=1\linewidth]{Plot_C_p_n5}
%			\caption{$n=5$}
%		\end{subfigure}
%		\begin{subfigure}{.5\textwidth}
%			\centering
%			\includegraphics[width=1\linewidth]{Plot_C_p_n6}
%			\caption{$n=6$}
%		\end{subfigure}
%		\caption{The graphs of $C_n: \mathbb{R} \to \mathbb{R}$ for $p(z)=z(z^n-1).$}
%		\label{Graph}
%	\end{figure}
	%%%%%%%%%%%%%%%%%%%%%%%%%%%%%%%%%%%

	\begin{figure}[h!]
		\begin{subfigure}{.5\textwidth}
			\centering
			\includegraphics[width=1\linewidth]{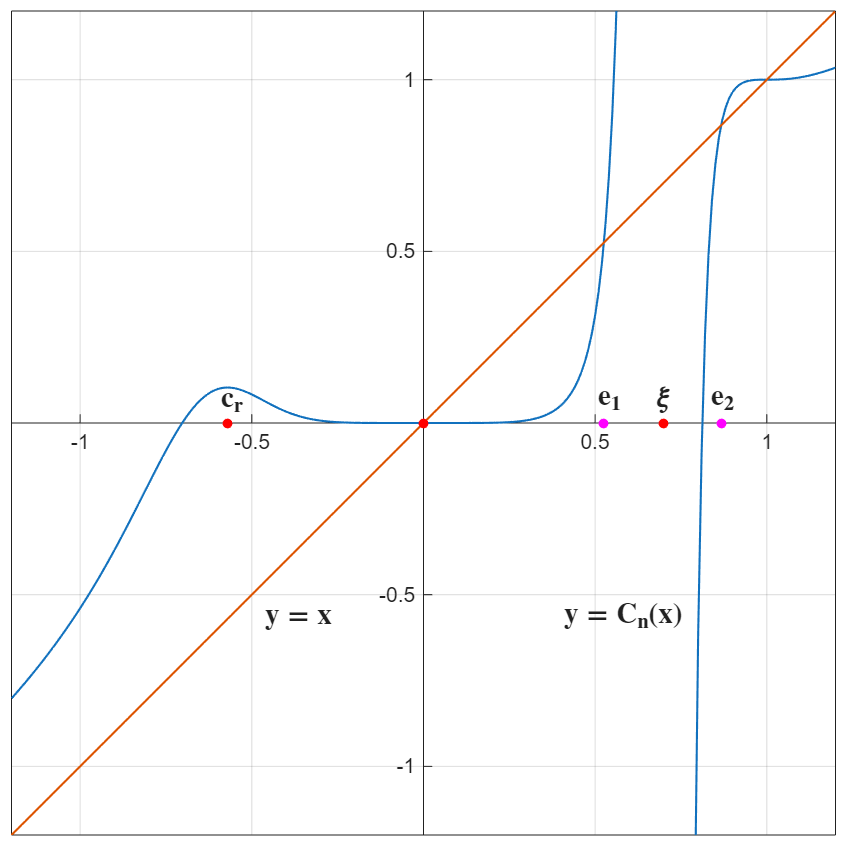}
			\caption{$n=5$}
		\end{subfigure}
		\begin{subfigure}{.5\textwidth}
			\centering
			\includegraphics[width=1\linewidth]{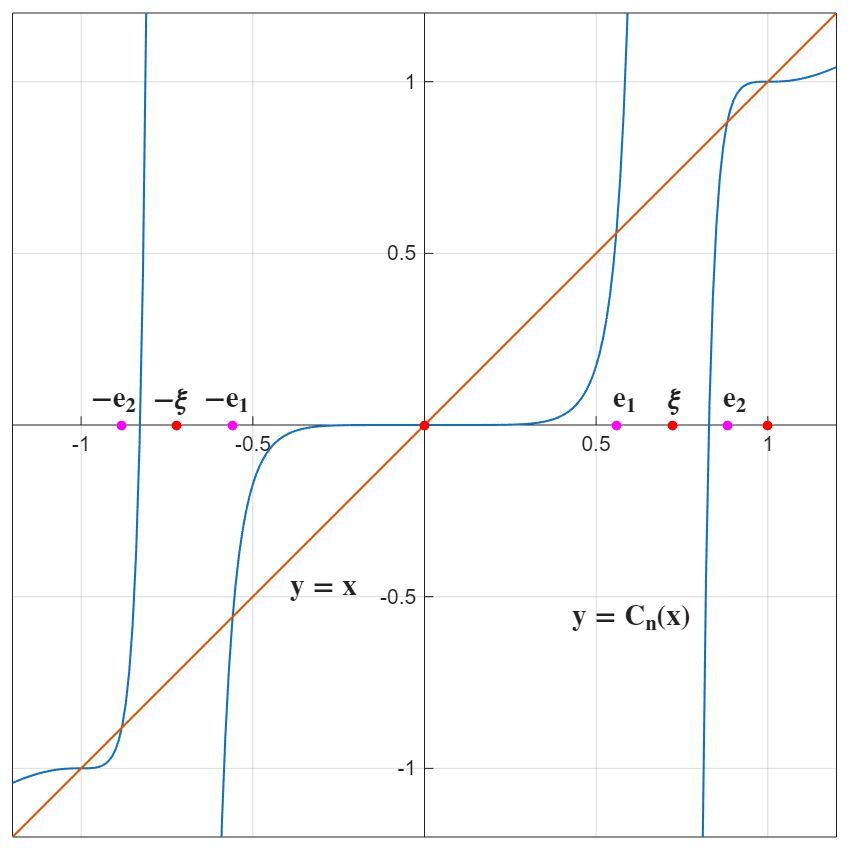}
			\caption{$n=6$}
		\end{subfigure}
		\caption{Graphs of $C_n: \mathbb{R} \to \mathbb{R}$.}
		\label{Graph}
	\end{figure}
	%%%%%%%%%%%%%%%%%%%%%%%%%%%%%%%%%%%
%	\begin{figure}[h!]
%		\centering
%		
%		\begin{subfigure}{0.8\textwidth}
%			\centering
%			\includegraphics[width=\linewidth]{Plot_C_p_n5}
%			\caption{$n=5$}
%		\end{subfigure}
%		
%		\vspace{0.3cm}
%		
%		\begin{subfigure}{0.8\textwidth}
%			\centering
%			\includegraphics[width=\linewidth]{Plot_C_p_n6}
%			\caption{$n=6$}
%		\end{subfigure}
%		
%		\caption{Graphs of $C_n:\mathbb{R}\to\mathbb{R}$ for $p(z)=z(z^n-1)$.}
%	\end{figure}
	%%%%%%%%%%%%%%%%%%%%%%%%%%%%%%%%%%%%%%%
\subsection{Critical points}  The derivative of $C_n$ is given by
\begin{equation}\label{Deri-C_n}
	C_n'(z)=\frac{n(n+1)z^n(z^n-1)^2((n+1)(2n+1)z^n+(n-1))}{2((n+1)z^n-1)^4}.
\end{equation}
Since the degree of $C_n$ is $3n+1$, it has $6n$ critical points, counted with multiplicity. The following lemma describes them all, and its proof follows directly from Equations~\eqref{eq:Cp2} and~\eqref{Deri-C_n}.
	\begin{lem}[Critical points]
	\begin{enumerate}
		\item The  origin is a critical point  of order $n$.
		\item Each $n$-th root of unity (i.e., non-zero root of $p$)  as well as each $n$-th root of $\frac{1}{n+1}$ (i.e., pole of $C_n$) is a critical point of order $2$.
		\item Each solution of the following equation is a simple critical point.
		\begin{equation}\label{free_cr_eqn}
			(n+1)(2n+1)z^n+(n-1)=0.
		\end{equation} 
	\end{enumerate}  
	\end{lem}
	
	None of the solutions of Equation~\eqref{free_cr_eqn} is a root of $p$ or a pole of $C_n$. We refer to these points as the \emph{free critical points} of $C_n$.  Note that all the free critical points have the same modulus, and their arguments  are $\frac{(2j+1)\pi}{n}$ for $j=0,1,2, \dots, n-1$. Indeed, if  $c$ is a free critical point, then 
	\begin{equation}
		 c^n=\frac{-(n-1)}{(n+1)(2n+1)}.
		 \label{freecriticalpoint}
	\end{equation} 
	Let 
	\begin{equation*}
		 c_0 := \left( \frac{n-1}{(n+1)(2n+1)}\right)^{\frac{1}{n}}\exp\left(\frac{i \pi}{n}\right).
	\end{equation*}  
 Then, using the notation above, Equation~\eqref{Deri-C_n} can be rewritten as 
\begin{equation}
	C_n '(z) =\frac{n(n+1)^2(2n+1)z^n(z^n-1)^2 (z^n-c_0 ^n)}{2((n+1)z^n-1)^4}.
	\label{Deri-C_n-reworded}
\end{equation}
It is observed from Equation~\eqref{free_cr_eqn} that there is no real free critical point of $C_n$ whenever $n$ is even. However, for $n$ odd, there is unique real critical point, say $c_r$ and that is negative. Equation~\eqref{Deri-C_n-reworded} leads to the following lemma.
\begin{lem}
	For each even $n$, $C_n$ is increasing in $\mathbb{R} \setminus \{\pm \xi\}$ where $\xi$ is the real positive pole. For each odd $n$, $C_n$ is increasing in $(-\infty, c_r) \cup (0, \xi) \cup (\xi, +\infty)$, and decreasing in $(c_r, 0)$.
	\label{Cn-increasing-decreasing}
\end{lem}
The next lemma describes the location of all the free critical points.
\begin{lem}(Critical lines) Let $L_k = \{t \exp(\frac{k \pi i}{n}): t \in \mathbb{R}\}$ for $k=0,1,2,\cdots, n-1$.
	\begin{enumerate} 
	   \item If $n$ is odd, then each $ L_k$  contains exactly one free critical point. In fact, the free critical point is on $L_k ^+ := \{t \exp(\frac{k \pi i}{n}): t >0\}$.
	   \item If $n$ is even, then there are exactly two free critical points on $L_k$ for each odd $k$ and that are symmetric about the origin.
	\end{enumerate}
	\label{criticallines}
\end{lem}
 We refer to the lines containing the free critical points as \textit{critical lines}. Observe that there are $\frac{n}{2}$ critical lines when $n$ is even, whereas for odd $n$, there are $n$ critical lines.
The cases $n=5$ and $n=6$ are illustrated in Figure~\ref{L_k}. 
 
%%%%%%%%%%%%%%%%%%%%%%%%%%%%%%%%%%%%%%%%%%%%%%%%%%%%%%%%%%%%%%%%%%%
%%%%%%%%%%%%%%%%%%%%%%
%%%%%%%%%%%%%%%%%%%%%%
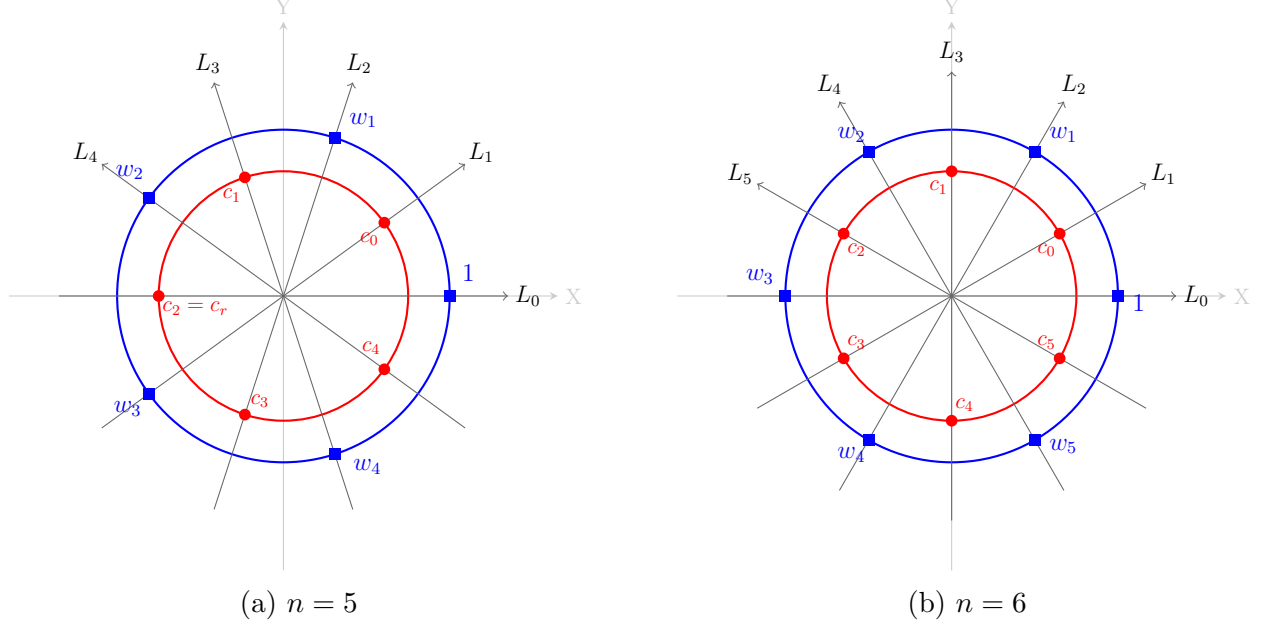
\begin{figure}[h!]
	\centering
	% ------------------------------------------------------------------
	% LEFT DIAGRAM (5-Root Geometry System - 5 Lines)
	% ------------------------------------------------------------------
	\begin{subfigure}[b]{0.48\textwidth}
		\centering
		\begin{tikzpicture}[scale=2.2]
			% Axes
			\draw[->, >=stealth, gray!40, thin] (-1.65,0) -- (1.65,0) node[right, scale=0.7] {X};
			\draw[->, >=stealth, gray!40, thin] (0,-1.65) -- (0,1.65) node[above, scale=0.7] {Y};
			
			% Concentric Circles
			\draw[blue, thick] (0,0) circle (1.0cm);
			\draw[red, thick] (0,0) circle (0.75cm);
			
			% 5 Full Lines (L_0 to L_4) across the origin, labeled only on the upper half
			\foreach \k in {0,1,2,3,4} {
				\def\angle{\k * 36}
				\draw[->,white!40!black, thin] (\angle-180:1.35) -- (\angle:1.35);
				\node[scale=0.75] at (\angle:1.47) {$L_{\k}$};
			}
			
			% Outer points (w_k) -> converted to blue squares
			\fill[blue] (0:1) +(-1pt,-1pt) rectangle +(1pt,1pt) node[above right, scale=0.8, xshift=-1pt] {$1$};
			\fill[blue] (72:1) +(-1pt,-1pt) rectangle +(1pt,1pt) node[above right, scale=0.8, yshift=-2pt] {$w_1$};
			\fill[blue] (144:1) +(-1pt,-1pt) rectangle +(1pt,1pt) node[above left, scale=0.8, xshift=-1pt, yshift=2pt] {$w_2$};
			\fill[blue] (216:1) +(-1pt,-1pt) rectangle +(1pt,1pt) node[below left, scale=0.8, xshift=-2pt, yshift=-2pt] {$w_3$};
			\fill[blue] (288:1) +(-1pt,-1pt) rectangle +(1pt,1pt) node[below right, scale=0.8, xshift=2pt, yshift=-1pt] {$w_4$};
			
			% Inner points (c_k) - Maintained as red circles
			\fill[red] (180:0.75) circle (1pt) node[below right, scale=0.7, xshift=-2pt, yshift=1pt] {$c_2=c_r$};
			\fill[red] (252:0.75) circle (1pt) node[above right, scale=0.7, xshift=-1pt] {$c_3$};
			\fill[red] (324:0.75) circle (1pt) node[above left, scale=0.7, xshift=2pt, yshift=4pt] {$c_4$};
			\fill[red] (36:0.75) circle (1pt) node[below left, scale=0.7, xshift=1pt, yshift=-1pt] {$c_0$};
			\fill[red] (108:0.75) circle (1pt) node[below left, scale=0.7, xshift=2pt, yshift=-2pt] {$c_1$};
			
		\end{tikzpicture}
		\caption{$n=5$}
	\end{subfigure}
	\hfill
	% ------------------------------------------------------------------
	% RIGHT DIAGRAM (6-Root Geometry System - 6 Lines)
	% ------------------------------------------------------------------
	\begin{subfigure}[b]{0.48\textwidth}
		\centering
		\begin{tikzpicture}[scale=2.2]
			% Axes
			\draw[->, >=stealth, gray!40, thin] (-1.65,0) -- (1.65,0) node[right, scale=0.7] {X};
			\draw[->, >=stealth, gray!40, thin] (0,-1.65) -- (0,1.65) node[above, scale=0.7] {Y};
			
			% Concentric Circles
			\draw[blue, thick] (0,0) circle (1.0cm);
			\draw[red, thick] (0,0) circle (0.75cm);
			
			% 6 Full Lines (L_0 to L_5) across the origin, labeled only on the upper half
			\foreach \k in {0,1,2,3,4,5} {
				\def\angle{\k * 30}
				\draw[->,white!40!black, thin] (\angle-180:1.35) -- (\angle:1.35);
				\node[scale=0.75] at (\angle:1.47) {$L_{\k}$};
			}
			
			% Outer points (w_k) -> converted to blue squares
			\fill[blue] (0:1) +(-1pt,-1pt) rectangle +(1pt,1pt) node[below right, scale=0.8, yshift=2pt] {$1$};
			\fill[blue] (60:1) +(-1pt,-1pt) rectangle +(1pt,1pt) node[above right, scale=0.8, yshift=-2pt] {$w_1$};
			\fill[blue] (120:1) +(-1pt,-1pt) rectangle +(1pt,1pt) node[above left, scale=0.8, yshift=-2pt] {$w_2$};
			
			% Adjusted w_3 to 182 degrees (shifted slightly up/left away from L_3 line)
			\fill[blue] (180:1) +(-1pt,-1pt) rectangle +(1pt,1pt) node[above left, scale=0.8, xshift=-4pt, yshift=-1pt] {$w_3$};
			
			\fill[blue] (240:1) +(-1pt,-1pt) rectangle +(1pt,1pt) node[below left, scale=0.8, yshift=-2pt] {$w_4$};
			\fill[blue] (300:1) +(-1pt,-1pt) rectangle +(1pt,1pt) node[below right, scale=0.8, yshift=2pt] {$w_5$};
			
			% Inner points (c_k) - Maintained as red circles
			\fill[red] (30:0.75) circle (1pt) node[below left, scale=0.7, xshift=2pt, yshift=-1pt] {$c_0$};
			\fill[red] (90:0.75) circle (1pt) node[below left, scale=0.7, xshift=2pt, yshift=-1pt] {$c_1$};
			\fill[red] (150:0.75) circle (1pt) node[below right, scale=0.7, xshift=-2pt, yshift=-1pt] {$c_2$};
			\fill[red] (210:0.75) circle (1pt) node[above right, scale=0.7, xshift=-2pt, yshift=1pt] {$c_3$};
			\fill[red] (270:0.75) circle (1pt) node[above right, scale=0.7, xshift=-2pt, yshift=1pt] {$c_4$};
			\fill[red] (330:0.75) circle (1pt) node[above left, scale=0.7, xshift=2pt, yshift=1pt] {$c_5$};
			
		\end{tikzpicture}
		\caption{$n=6$}
	\end{subfigure}
	
	\caption{The roots of $p_n$ and the free critical points of $C_n$ are represented by boxes and dots, respectively. For $n=5$, the five free critical points are on the lines $L_k, k=0,1,2,3,4$, one on each line, whereas for $n=6, $ each of the lines $L_1, L_3,$ and $L_5$ contains exactly two free critical points and that are symmetric with respect to the origin.}
	\label{L_k}
\end{figure}
%%%%%%%%%%%%%%%%%%%%%%%%%%%%%%
	%%%%%%%%%%%%%%%%%%%%%%%%%%%%%%%%%%%%%%%%%%%%%%%%%%%%%%%%%%%	

 Let $c^*$ denote the critical value corresponding to a free critical point $c$. Then
 	$$c^*=C_n (c) =\frac{ c^{n+1}\big(n(n+1)(2n+1)c^{2n}-2n(n+1)c^n-n(n-1)\big)}{2\big((n+1)c^n-1\big)^3}.$$
	This, along with Equation~\eqref{freecriticalpoint}, gives
\begin{equation}
	 c^* =-\frac{(n-1)^2(2n+1)}{27(n+1)^2}c.
	 \label{critical-point-value}
\end{equation} 
There is a key estimate of the free critical points and the corresponding critical values.	
\begin{lem}\label{c_and_c^*}
		If	$c^*$ is the critical value corresponding to a free critical point $c$ of $C_n$, then
		$$|c^*| \left\{\begin{array}{ccc}
			<|c| & \text { whenever } & n \leq 16, \\
			>|c| & \text { whenever } & n \geq 17.
		\end{array}\right.
		$$
	
\end{lem}
	\begin{proof}
		By Equation~\eqref{critical-point-value}, $|c^*|=D_n |c|$, where $D_n=\frac{(n-1)^2(2n+1)}{27(n+1)^2}$. Define $T(n) =(n-1)^2(2n+1)-27(n+1)^2 $. Since  $T(16)=-378<0$ and $T(17)=212>0$, there exists a real root of $T$ in the interval $(16,17)$.
		Moreover, $T'(n)=6(n^2-10n-9)=6\big(n-(5-\sqrt{34})\big)\big(n-(5+\sqrt{34})\big)$.
		Hence, $T$ increases strictly in $(-\infty,5-\sqrt{34})$, attains a local maximum at $n=5-\sqrt{34}$, and then decreases strictly in $(5-\sqrt{34},5+\sqrt{34})$  attaining a local minimum at $n=5+\sqrt{34}$. Thereafter, it increases strictly. 
		
		As $T(5-\sqrt{34})=-796 + 136\sqrt{34} \approx -2.99 <0$, the local maximum of $T$ is negative. This implies that the root lying in the interval $(16,17)$ is the only real root of $T$. Consequently, $T(n)<0$ for all $n$ smaller than this root, and $T(n)>0$ for all $n$ larger than this root.
		In particular,
		$$
		T(n) \left\{\begin{array}{ccc}
			<0 & \text { for all } & n \leq 16, \\
			>0 & \text { for all } & n\geq 17.
		\end{array}\right.
		$$
		Equivalently,
		$$
		D_n \left\{\begin{array}{ccc}
			<1 & \text { whenever } & n \leq 16, \\
			>1 & \text { whenever } & n\geq 17,
		\end{array}\right.
	~\mbox{and therefore,}~
		|c^*| \left\{\begin{array}{ccc}
			<|c| & \text { whenever } & n \leq 16, \\
			>|c| & \text { whenever } & n\geq 17.
		\end{array}\right.
		$$
		This completes the proof.
	\end{proof}
	Recall that, for odd $n$, there is unique real free critical point $c_r$ of $C_n$.  The corresponding  critical value $c_r^*$  satisfies
\begin{equation}
	 (c_r^*)^n =\frac{(n-1)^{2n+1}(2n+1)^{n-1}}{27^n(n+1)^{2n+1}}.
	 \label{realcriticalvalue}
\end{equation}
For odd $n$, the relative location of real free critical value and extraneous fixed points is given in the next proposition.
	\begin{prop}\label{Odd_free_cr}
		Let $n$ be an odd natural number. Let $c_r^*$ be the critical value corresponding to the real free critical point $c_r$ and let $e_1$ and $e_2$ be the real extraneous fixed points with $e_1<e_2$.
		\begin{enumerate}
			\item If $n\leq 15$, then $c_r^*<e_1$, and if $n\geq 19$, then $c_r^*>e_2$.
			\item If $n=17$, then $c_r ^* > e_1$ and  $C_n (c_r ^*)>1.$
		\end{enumerate}
	\end{prop}
The next two lemmas hold for all $n>1$, and will be used to prove the Proposition~\ref{Odd_free_cr}. The first one analyses the variation of $(c_r ^*)^n$ for odd $n$, whereas the second captures the change of the $n$-th powers of the positive extraneous fixed points $e_1$ and $e_2$.
	\begin{lem}\label{fn_S}
		Let $S(n)=\frac{(n-1)^{2n+1}(2n+1)^{n-1}}{27^n(n+1)^{2n+1}},$ where $n$ is a positive integer. Then for $n\geq 5$, $S(n+1)>S(n)$. Moreover, the function $S$ satisfies the following relation: $S(2)>S(3)>S(4)>S(5)$.
	\end{lem}
	
	\begin{proof}
		Note that
%		$$S(n+1)=\frac{n^{2n+3}(2n+3)^n}{27^{n+1}(n+2)^{2n+3}},$$
%		and therefore
%		$$\frac{S(n+1)}{S(n)}=\frac{1}{27} \cdot \frac{n^{2n+3}}{(n-1)^{2n+1}} \cdot \frac{(2n+3)^n}{(2n+1)^{n-1}}\cdot\frac{(n+1)^{2n+1}}{(n+2)^{2n+3}}.$$
%		Rearranging the terms, we obtain
		$$\frac{S(n+1)}{S(n)}=\frac{n^2(2n+1)}{27(n+2)^2}\left(\frac{n(n+1)}{(n-1)(n+2)}\right)^{2n+1} \left(\frac{2n+3}{2n+1}\right)^n.$$
		Let
	 $A_n =\left(\frac{n(n+1)}{(n-1)(n+2)}\right)^{2n+1}$ and $ B_n =\left(\frac{2n+3}{2n+1}\right)^n.$ 
		Since
		 $\frac{n(n+1)}{(n-1)(n+2)}=1+\frac{2}{n^2+n-2},$ 
	the Bernoulli's inequality yields
		 $A_n>1+\frac{2(2n+1)}{n^2+n-2}.$ 
Similarly, using
		 $\frac{2n+3}{2n+1}=1+\frac{2}{2n+1}$ 
		and the Bernoulli's inequality, we have
		 $B_n>1+\frac{2n}{2n+1}.$ 
		Therefore,
		$$A_nB_n>\left(1+\frac{2(2n+1)}{n^2+n-2}\right)\left(1+\frac{2n}{2n+1}\right)= \frac{n(n+5)}{(n-1)(n+2)}   \frac{4n+1}{2n+1}.$$
 	 
		Hence,
		$$\frac{S(n+1)}{S(n)}>\frac{n^2(2n+1)}{27(n+2)^2}   \frac{n(n+5)}{(n-1)(n+2)}   \frac{4n+1}{2n+1}=\frac{n^3(n+5)(4n+1)}{27(n-1)(n+2)^3}.$$
	 
		Thus, it remains to show that
	 $n^3(n+5)(4n+1)>27(n-1)(n+2)^3,$ i.e.,  
	 $4n^5+21n^4+5n^3>27n^4+135n^3+162n^2-108n-216.$ 
		This is equivalent to $n^2(4n^3-6n^2-130n-162)+108n+216>0$.
%		$$
%		\begin{aligned}
%			&n^3(n+5)(4n+1)-27(n-1)(n+2)^3 \\
%			&=4n^5-6n^4-130n^3-162n^2+108n+216\\
%			&=n^2(4n^3-6n^2-130n-162)+108n+216.
%		\end{aligned}
%		$$
		It is straightforward to verify that the polynomial $4n^3-6n^2-130n-162$ is increasing for all $n \geq 7$ and is positive. Therefore, for $n\geq 7$, we have
		 $4n^5-6n^4-130n^3-162n^2+108n+216>0,$ 
		and hence
		$$S(n+1)>S(n) \quad \text{for all } n \geq 7.$$
		That  $S(n+1)>S(n) \quad \text{for } n=5,6$, and $S(2)>S(3)>S(4)>S(5)$ follow from their approximate values, as given in the table below. The graph of $S$ is given in Figure~\ref{Graphs-S-E_1-E_2}(a).
		
		\begin{table}[h]
			\centering
			\begin{tabular}{|c|c|c|c|c|c|}
				\hline
				$S(2)$ & $S(3)$ & $S(4)$ & $S(5)$ & $S(6)$ & $S(7)$ \\ \hline
				$2.82 \times 10^{-5}$ & $1.94 \times 10^{-5}$ & $1.38  \times 10^{-5}$  & $1.18  \times 10^{-5}$ & $1.21  \times 10^{-5}$ & $1.45   \times 10^{-5}$  \\ \hline
			\end{tabular}
		\end{table}
\end{proof}		
\begin{figure}[h!]
\begin{subfigure}{.5\textwidth}
\centering
\includegraphics[width=1\linewidth]{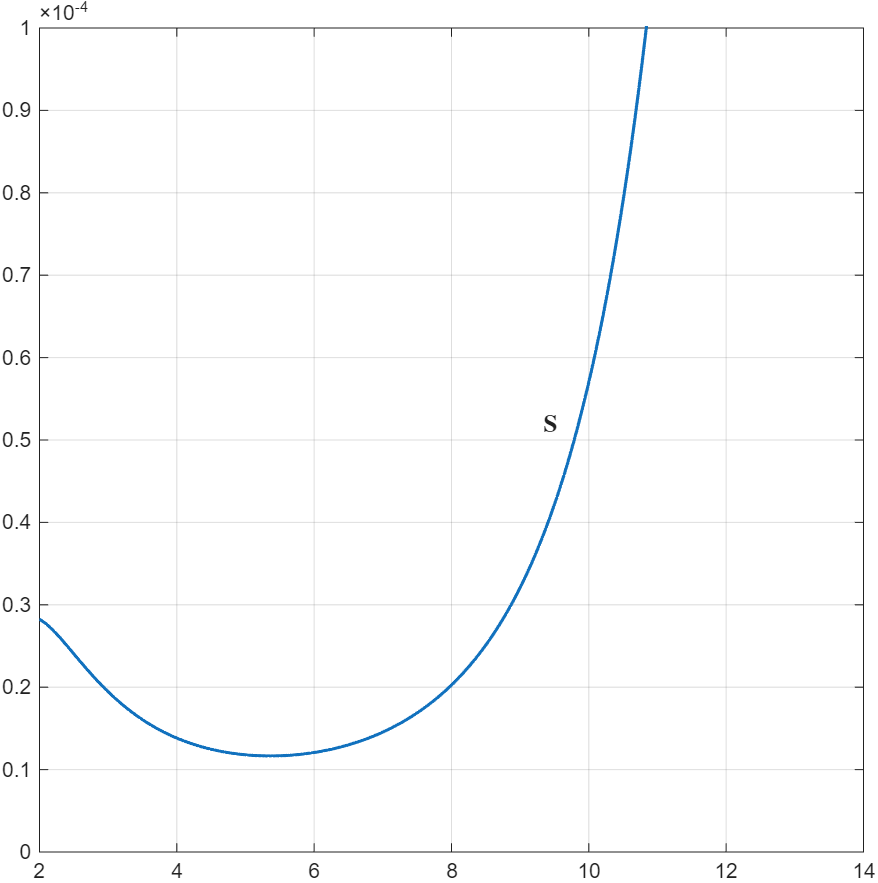}
\caption{The minimum value of $S$ is being attained at $n=5$.}
\end{subfigure}
\begin{subfigure}{.5\textwidth}
\centering
\includegraphics[width=1\linewidth]{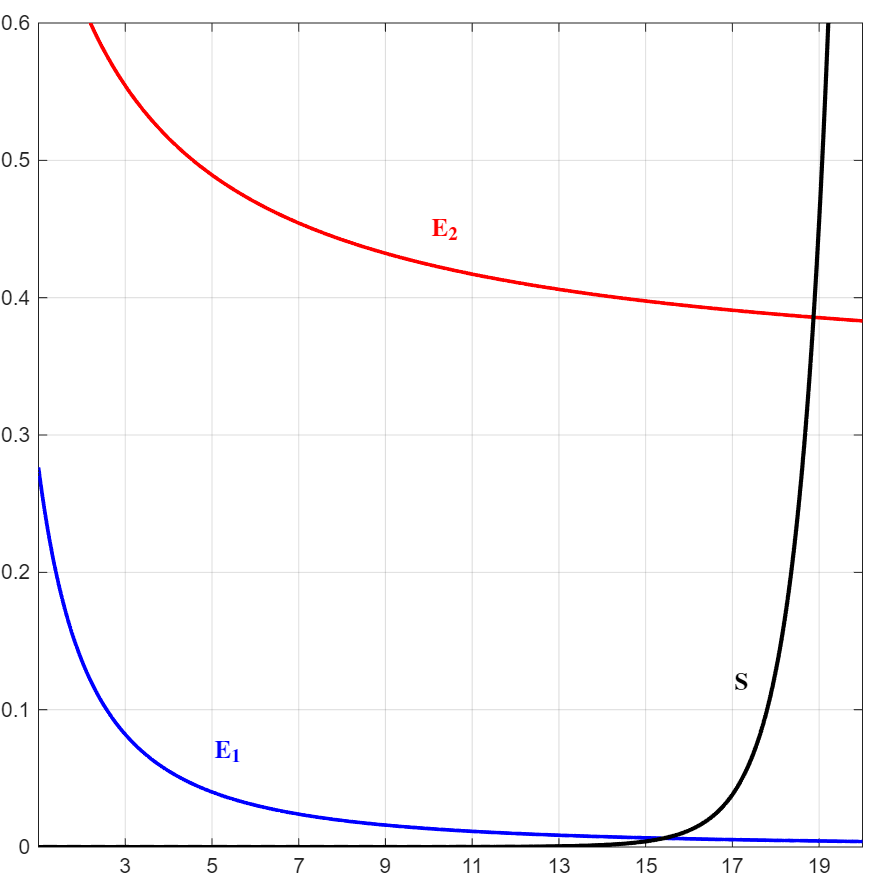}
\caption{ $S(n)<E_1(n)$ for $n\leq 15$ and $S(n)>E_2(n)$ for $n\geq 17$.}
\end{subfigure}
\caption{Graphs of $S, E_1, E_2: \mathbb{R} \to \mathbb{R}$.}
\label{Graphs-S-E_1-E_2}
\end{figure}
%%%%%%%%%%%%%%%%%%%%%%%%%%%%%%
\begin{lem}\label{E_decreasing}
The sequences  
$$E_1(n):=\frac{(n+1)(n+4)-n\sqrt{(n+1)(n+9)}}{2(n+1)(3n+2)}$$  
		and $$ E_2(n):=\frac{(n+1)(n+4)+n\sqrt{(n+1)(n+9)}}{2(n+1)(3n+2)} $$ 
	 are strictly decreasing  for $n \geq 1$.
	\end{lem}
	
	\begin{proof}
		We establish the strict monotonicity of both sequences by algebraic simplification via rationalization of their respective numerators and denominators.
		 
		Multiplying both the numerator and the denominator of $E_1(n)$ by the expression $ (n+1)(n+4) + n\sqrt{(n+1)(n+9)}$ and simplifying, we get 
%			$$E_1(n) = \frac{\left((n+1)(n+4)\right)^2 - \left(n\sqrt{(n+1)(n+9)}\right)^2}{2(n+1)(3n+2) \left((n+1)(n+4) + n\sqrt{(n+1)(n+9)}\right)}.$$
%			
%			By simplifying the numerator, we get
%			$$E_1(n) = \frac{8(3n+2)(n+1)}{2(n+1)(3n+2) \left((n+1)(n+4) + n\sqrt{(n+1)(n+9)}\right)},$$
%			which simplifies to
			 $E_1(n) = \frac{4}{(n+1)(n+4) + n\sqrt{n^2+10n+9}}.$ 
%			Let $D_1(n) = (n+1)(n+4) + n\sqrt{n^2+10n+9}$ be the denominator of $E_1(n)$. Since each individual component $(n+1)$, $(n+4)$, $n$, and $\sqrt{n^2+10n+9}$ is strictly positive and strictly increasing for all $n \geq 1$, their sum and product $D_1(n)$ must be strictly increasing. 
%			
%			Since $E_1(n) = \frac{4}{D_1(n)}$ has a constant positive numerator and a strictly increasing denominator,
This is clearly a strictly decreasing sequence for $n \geq 1$.
			
		 Now, by multiplying both the numerator and the denominator of $E_2(n)$ by $ (n+1)(n+4) - n\sqrt{(n+1)(n+9)}$, we get
			 $E_2(n) = \frac{4}{(n+1)(n+4) - n\sqrt{n^2+10n+9}}.$ 
			
			Letting $D_2(n) = (n+1)(n+4) - n\sqrt{n^2+10n+9}$, it is observed that
			$$D_2'(n) = 2n+5 - \frac{2n^2+15n+9}{\sqrt{n^2+10n+9}}.$$ 
			
			We have $D_2'(n)>0$ if and only if $(2n+5)\sqrt{n^2+10n+9}>2n^2+15n+9$. Since both sides are positive for $n>1$, squaring gives $(2n+5)^2(n^2+10n+9)>(2n^2+15n+9)^2$, which simplifies to $10n+9>0$. This condition holds for all $n$, giving that $D_2'(n)>0$ for all $n$. Thus, $D_2(n)$ is a strictly increasing sequence and therefore $E_2(n)$ is strictly decreasing. 	\end{proof}	
	Now we present the proof of Proposition~\ref{Odd_free_cr}.
	\begin{proof}[Proof of Proposition~\ref{Odd_free_cr}]
	\begin{enumerate}
		\item For $n=3$, $c_r^*=0.0268925$ and $e_1= 0.434436$. Thus, the result follows directly. 
		
		Now consider $n>3$. Since $c_r^*$, $e_1$, and $e_2$ are positive, it is sufficient to prove the inequalities by considering the $n$-th power of these numbers. Note that $S(n)=(c_r^*)^n$, $E_1(n)=e_1^n$, and $E_2(n)=e_2^n$, where $S(n),E_1(n)$, and $E_2(n)$ are as defined in Lemmas~\ref{fn_S} and~\ref{E_decreasing}, respectively. Thus, it is enough to prove that
		 $$S(n)<E_1(n)\quad \text{for all odd } n \text{ with } 5 \leq n\leq 15,~\mbox{and }
		  S(n)>E_2(n)\quad \text{for all odd } n\geq 19.$$ 
		  This is demonstrated in Figure~\ref{Graphs-S-E_1-E_2}(b).
		Note that
		$$S(15)=\frac{14^{31}  31^{14}}{27^{15}  16^{31}}\approx 0.0040818,~\mbox{and}~
		 E_1(15)=\frac{(16)(19)-15\sqrt{(16)(24)}}{2(16)(47)} \approx 0.0066896.$$
		Hence $S(15)<E_1(15).$
		From Lemmas~\ref{fn_S} and~\ref{E_decreasing}, we have $S(n)$ is increasing in $\{m \in \mathbb{N}:m \geq 5 \}$ and $E_1(n)$ is decreasing.
		Thus, for every odd $n$ with $5\leq n\le15$, we have
		$$S(n)\le S(15)<E_1(15)\le E_1(n).$$
		Consequently,
		$$S(n)<E_1(n) \qquad \text{for all odd } n\le15.$$
		Since   $S(19)=\frac{18^{39} 39^{18}}{27^{19} 20^{39}} \approx 0.455732,$ 
		and	$E_2(19)=\frac{(20)(23)+19\sqrt{(20)(28)}}{2(20)(59)} \approx 0.385433,$ we have
		$$S(19)>E_2(19).$$
		It follows from Lemmas~\ref{fn_S} and~\ref{E_decreasing} that $S(n)$ is increasing on $\{m \in \mathbb{N}:m \geq 19 \}$ and $E_2(n)$ is decreasing. 
		Therefore, for every odd $n\ge 19$,
		$$S(n)\ge S(19)>E_2(19)\ge E_2(n).$$
		Hence,
		$$S(n)>E_2(n) \qquad \text{for all odd } n\ge 19.$$
		\item 	For $n=17$, we have $(e_1 )^{17} \approx 0.005$ and  $(c_r ^*)^{17}\approx 0.038 $ (see Equations~\eqref{ext_f1} and~\eqref{realcriticalvalue}). In other words $c_r ^* > e_1$. Letting $s = (c_r ^*)^{17}$, a direct numerical evaluation yields $s \approx 0.0381563$. Substituting this value into the expression for $C_n (c_r ^*)$, we have
		$$\big(C_n(c_r ^*)\big)^{17} = \frac{17^{17} s^{18} (630s^2 - 36s - 16)^{17}}{2^{17} (18 s - 1)^{51}}.$$
		Evaluating the individual terms with $s \approx 0.0381563,$ we get 
		\begin{align*}
			18s - 1  \approx -0.313187 ~\mbox{and}~
			630 s^2 - 36s - 16  \approx -16.4564.
		\end{align*}
		Substituting these values directly into the above expression yields
		$$ \big(C_n(c_r^*)\big)^{17}  \approx \frac{17^{17} (0.0381563)^{18}  (16.4564)^{17}}{2^{17}  (0.313187)^{51}} \approx 4.57 \times 10^{36}>1.$$
		%    Since $1.1202 \times 10^{37} > 1$, we firmly establish that $\big(C_n(c^*)\big)^{17} > 1$.
	\end{enumerate}	
		In particular,  $C_n (c_r ^*) >1$.
	\end{proof}
	
	\subsection{Symmetry of $C_n$ and its dynamics on lines}    
This subsection describes various symmetries exhibited by the map $C_n$.
We begin with an immediate consequence of Lemma~\ref{Sym_C_p}.
\begin{lem}(Rotational Symmetry)
	For each $n>1$, 
	$\{z\mapsto \lambda z: \lambda^n =1\}\subseteq\Sigma C_n.$
	\label{rot-symm}\end{lem} 
The map $C_n$ also possesses reflection symmetries, as described below.
\begin{lem}(Symmetry about  lines) \label{Inv_lines_ref-sym}
	For each $k \in \{ 0, 1, \dots, n-1\}$, let $\lambda_{k} = \exp\left(\frac{k\pi i}{n}\right)$ and  $L_{k} = \left\{ \lambda_k t : t \in \mathbb{R}\right\}$. Then the following statements hold.
	\begin{enumerate}
		\item The map $C_{n}$ satisfies $C_n (z)=\lambda_k ^{-2} \overline{(C_n (\lambda_k ^2 \overline{z}))}$, i.e., $C_n$ is symmetric with respect to the line $L_k$ for each $k$. In particular, each line $L_k$ is invariant under $C_{n}$. 
		\item If  $k$ is odd, then 
	$C_{n}(\lambda_{k} z) = \lambda_{k} Q_n (z)$ for all $z,$  where 
	\begin{equation}
		Q_n(z)= \frac{z^{n+1} \big(n(n+1)(2n+1)z^{2n} + 2n(n+1)z^{n} - n(n-1)\big)}{2\big((n+1)z^{n} + 1\big)^3}.
		\label{Q_conj} 
	\end{equation} 
		\item 	If $k$ is even, then 
		$C_{n}(\lambda_{k} z) = \lambda_{k} C_{n}(z)$ for all $z.$ 
	
	\end{enumerate}  
\end{lem}
\begin{proof}
	Let $\lambda_{k} = \exp\left(\frac{k\pi i}{n}\right)$ for $k = 0, 1, \dots, n-1$. Then
	$$
	\lambda_{k}^{n}= \left\{\begin{array}{ccc}
		1 & \text { if } k \text { is even}, \\
		-1 & \text { if } k \text { is odd}.
	\end{array}\right.
	$$
	\begin{enumerate}
		\item 	It follows from Equation~\eqref{eq:Cp2} that $C_n (z)=\lambda_k ^{-2} \overline{(C_n (\lambda_k ^2 \overline{z}))}$. Since the map $z \mapsto \lambda_k ^2 \overline{z}$ is the reflection about the line $L_k$, the map $C_n$ is symmetric with respect to $L_k$. In particular, the line $L_k$ is invariant under $C_n$. 
	
		\item If $k$ is odd then $\lambda_k ^n =-1$ and consequently, 	$$
		C_{n}(\lambda_{k}z) = \frac{\lambda_{k}z^{n+1} \big(n(n+1)(2n+1)z^{2n} + 2n(n+1)z^{n} - n(n-1)\big)}{2\big((n+1)z^{n} + 1\big)^3} = \lambda_{k} Q_n(z),
		$$
		where		
		$ Q_n(z) $ is as defined in the statement of this lemma.
			\item If $k$ is even then $\lambda_k ^n =1$ and from Equation~\eqref{eq:Cp2}, we have   $C_{n}(\lambda_{k} z) = \lambda_{k} C_{n}(z).$ 
	\end{enumerate}	 
\end{proof}

 The following remarks are immediate.

\begin{rem}
	For $k=0$, Lemma~\ref{Inv_lines_ref-sym}(1) gives that 
	$C_n (z)=\overline{C_n(\overline{z})}$, i.e., $C_n$ is symmetric about the real line.
	\label{symm-R}
\end{rem}
Here is an observation on the map $Q_n$, defined in Equation~\eqref{Q_conj}.
\begin{rem}
	\begin{enumerate}
	\item 
If $n$ is odd, then  $Q_n (z)=-C_n(-z)$  for all $z$.
If $n$ is even, then $Q_n$ is an odd function, i.e., $Q_n (z)=-Q_n(-z)$  for all $z$. 
\item For all $x>0$, the denominator of $Q_n (x)-x$ is positive, whereas its  numerator can be seen to be $ x((n+1) (-3n -2)x^{3n} +(n+1)(-4n -6)x^{2n} +(-n^2-5n -6)x^n)$, which is negative. Therefore $Q_n (x) < x$. 
\end{enumerate}
 \label{Qn-obs}
\end{rem}	
For each $n>1$, Lemma~\ref{Inv_lines_ref-sym} shows that the lines $L_k =\{t \exp(\frac{k \pi i}{n}): t \in \mathbb{R}\}$, for $k=0, 1, 2, \dots, n-1$, are invariant under $C_n$. We refer to them as \textit{invariant lines}. The following remark describes the behavior of $C_n$ on these  lines, which follows from Lemma~\ref{Inv_lines_ref-sym}(2), (3) and Remark~\ref{Qn-obs}. We say an invariant line $L_k$ is \textit{even-numbered} if $k$ is even. Otherwise, it is called \textit{odd-numbered}. In the following remark,  $\lambda_k$ and $L_k$ are as defined in the previous lemma.
\begin{rem}
	\begin{enumerate}
		\item If $n$ is odd, then there are $[\frac{n}{2}]+1$ even-numbered invariant lines and  $[\frac{n}{2}]$ odd-numbered invariant lines. Furthermore, the following statements are true.
		\begin{enumerate}
				\item 
				If $k$ is odd and $z \in L_k$, then $C_n (z)=\lambda_k Q_n(\lambda_k ^{-1}z)$. Moreover, it follows from Remark~\ref{Qn-obs}(1) that  
				$C_n (z)=-\lambda_k Q_n(-\lambda_k ^{-1}z)$.
				
				\item If $k$ is even and $z \in L_k$, then $C_n (z)= \lambda_k C_n(\lambda_k ^{-1}z)$. 				
				
		\end{enumerate}
		\item  If $n$ is even, then there are exactly $\frac{n}{2}$ even as well as odd-numbered invariant lines. Furthermore, the following statements are true.
		
		\begin{enumerate}
				\item 	If $k$ is odd and $z \in L_k$, then $C_n (z)= \lambda_{k} Q_{n}(\lambda_{k}^{-1} z) $.
				\item If $k$ is even and $z \in L_k$, then $C_n (z)=\lambda_k C_n(\lambda_k ^{-1}z)$.
				
		\end{enumerate}
	\end{enumerate}
	\label{dynamics-invariantlines}
\end{rem}	
	Recall that the critical lines of $C_n$ are the invariant lines passing through the origin that contain the free critical points. If $n$ is odd, then there are $n$ critical lines, each containing exactly one free critical point. If $n$ is even, then there are exactly $\frac{n}{2}$ critical lines, each containing exactly two free critical points.
	
	Note that for even $n$, the critical lines do not contain any root of unity, leading to the following useful observation.
	
	\begin{rem}\label{No_free_cr_for_even}
		Whenever $n$ is even, the forward orbits of the free critical points cannot converge to a non-zero real number. Consequently, the immediate basin of any non-zero root of $p_n$ under $C_n$  cannot contain a free critical point.
	\end{rem}
	 For a fixed $n$, it follows from Remark~\ref{dynamics-invariantlines} that $C_n :L_k \to L_k$ is conformally conjugate to  $Q_n : \mathbb{R} \to \mathbb{R}$ (or its conjugate under $z \mapsto -z$) or to $C_n: \mathbb{R} \to \mathbb{R}$  if $k$ is odd or even, respectively. In order to understand the dynamics of $C_n$ on the critical lines, it is enough to study $C_n$ or $Q_n$ as a function from $\mathbb{R} $ into itself. The following proposition establishes some properties of these maps that are required for this purpose.
\begin{prop}[Dynamics on critical lines]
For each $n$, $\lim\limits_{m \to \infty} C_n ^m (x)=1$ for every $x>e_2$ where $e_2$ is the bigger (real) extraneous fixed point of $C_n$ (see Lemma~\ref{extraneous-general}). 
 Further, for $n \leq 16$, the following hold.
\begin{enumerate}
	
	\item If $n$ is odd, then $\lim\limits_{m \to \infty} C_n ^m (x)= 0$ for all $x \in (-\infty, e_1)$.
	\item If $n$ is even, then $ \lim\limits_{m \to \infty} Q_n ^m (x)= 0$ for all 
	$x \in \mathbb{R}$.
\end{enumerate}
\label{dynamics-criticallines}
\end{prop}
 \begin{proof} 	
 	The function $C_n:[e_2,\infty) \to [e_2,\infty)$ is strictly increasing for each $n$ by Lemmas~\ref{extraneous} and~\ref{Cn-increasing-decreasing}. Moreover, by Lemma~\ref{C_p-real-all}(1), we have $C_n(x)<x$ for all $x\in (1,\infty)$. Therefore, for every $x \in (1,\infty)$, the sequence of iterates $\{C_n^m(x)\}_{m \geq 0}$ is strictly decreasing and bounded below by $1$. Hence, by the Monotone Convergence Theorem, it is convergent. Since the limit of this sequence must be a fixed point of $C_n$ and $1$ is the only fixed point in the interval $[1,\infty)$, its limit must be $1$. Now consider the map $C_n:[e_2,1] \to [e_2,1]$. Again, by Lemmas~\ref{extraneous} and~\ref{Cn-increasing-decreasing}, $C_n$ is strictly increasing on this interval, and by Lemma~\ref{C_p-real-all}(2), $C_n (x) >x$ for all $ x \in (e_2, 1)$. Consequently, for every $x \in (e_2, 1)$, the sequence $\{C_n^m(x)\}_{m \geq 0}$ is strictly increasing and bounded above by $1$. It follows by a similar argument that  $\lim_{m \to \infty} C_n ^m (x)= 1$ for $x \in (e_2, 1)$. This proves the first part of this proposition.
 
    The remainder of the proof is divided into two cases depending on the parity of $n$.
%	Note that  $n\leq 16$ if and only if $|C_n(c)|<|c|$ for every free critical point $c$ of $C_n$. 
	\begin{enumerate}
		\item For odd $n$, recall that $e_1$ is the smaller positive extraneous fixed point of $C_n$ (see Lemma~\ref{extraneous-general}). 
		The function $C_n : (0, e_1) \to (0, e_1)$ is increasing (by Lemma~\ref{Cn-increasing-decreasing}) and satisfies $C_n (x) <x$ (by Lemma~\ref{C_p-real-all}(1)). This gives that $\lim_{m \to \infty} C_n ^m (x)=0$ for all $x \in (0,e_1)$ (see illustration for $C_5$ in Figure~\ref{Graph}(a)). 
		Since $n \leq 16 $ is odd, it follows from Proposition~\ref{Odd_free_cr} that the critical value $c_r^{*} $ corresponding to the real critical point $c_r$ is in $(0,e_{1})$. This implies that $C_{n}\big((x_r ,0)\big) \subset  (0, e_1)$, where $x_r$ is the unique negative root of $C_n$. This is because $C_n$ is increasing in $(-\infty, c_r)$ and decreasing in $(c_r, 0)$ (see Lemma~\ref{Cn-increasing-decreasing}), satisfies  $C_n (c_r)>0$ (by Equation~\eqref{critical-point-value}) and $\lim_{x \to -\infty}C_n (x) = -\infty$ (see Equation~\eqref{eq:Cp2}). 
		\par Now  $C_{n}(x) > x$ for all $x \in (-\infty, x_r)$ and $C_n$ is increasing in this interval. For  a point $x \in (-\infty, x_r)$, if $C_n ^m (x)< x_r$ for all $m$ then $\{C_n ^m(x)\}_{m >0}$ being an increasing sequence must converge and its limit should be a fixed point of $C_n$. Since there is no fixed point in $(-\infty, x_r)$, there exists an $m_0$  such that $C_n ^{m_0}(x) > x_r$. Taking $m_0$ to be the smallest such number, we get that $C_n ^{m_0}(x) \in (x_r, 0)$ and therefore $C_n ^{m_0+1}(x) \in (0, e_1)$. By the conclusion obtained in the first part of the previous paragraph, this implies that $\lim_{m \to \infty} C_n ^m (x)= 0$ for all $x \in (-\infty, e_1)$. 
			
\item  Let $n$ be even. Recall from  Equation~\eqref{Q_conj} that $$Q_n (x)= \frac{x^{n+1} \big(n(n+1)(2n+1)x^{2n} + 2n(n+1)x^{n} - n(n-1)\big)}{2\big((n+1)x^{n} + 1\big)^3}.$$ Then, 
\begin{equation}\label{Q-deri}
Q_n'(x) = \frac{n(n+1)x^{n}(x^{n}+1)^{2} \big((n+1)(2n+1)x^{n} - (n-1)\big)}{2\big((n+1)x^{n}+1\big)^4},
\end{equation}
and hence $Q_n$ has exactly two real critical points other than $0$, namely the real $n$-th roots of $\frac{n-1}{(n+1)(2n+1)}$. Denote the positive critical point of $Q_n$ by $\tilde{c}$. The function $Q_n$ decreases in $(0, \tilde{c})$ and increases thereafter (see Figure~\ref{Graph_Q}(a) for $n =16$). Since $n \leq 16$, or equivalently $|C_n (c)| < |c|$ for every free critical point (see Lemma~\ref{c_and_c^*}) and $|c|=\tilde{c}$, we have $|C_n(c)|< \tilde{c}$. The free critical point $c$ is on an odd-numbered invariant line (by Lemma~\ref{criticallines}(2)) and  it follows from Remark~\ref{dynamics-invariantlines}(2(a)) that $C_n (c)=\lambda_k Q_n (\lambda_k ^{-1}c)$. Therefore, $|C_n(c)|=|Q_n(\tilde{c})|$ as $c=\lambda_k \tilde{c}$ for some odd $k$. Since $Q_n(\tilde{c})<0$, $-Q_n(\tilde{c})<\tilde{c}$, i.e., $Q_n(\tilde{c})> -\tilde{c}$.  In fact, it follows that $Q_n ^2 (x)<x$ for all $0< x< \tilde{c}$, and consequently, $\lim_{k \to \infty} Q_n ^{2k}(x)=0$. It also follows that this is true for all the points of $(-\tilde{x}_r, \tilde{x}_r)$, where $\tilde{x}_r$ is the positive $n$-th root of $\frac{1}{2n+1}\left(-1+\frac{n\sqrt{2}}{\sqrt{n+1}} \right)$, the unique positive root of $Q_n$. Note that $Q_n (x)<x$ for all $x> \tilde{x}_r$  by Remark~\ref{Qn-obs}(2) and $Q_n$  is increasing in $(\tilde{x}_r, \infty)$. For every $x$ in this interval, $\{Q_n ^m (x)\}_{m>0}$ is a decreasing sequence. It follows by arguing as in the first part of this proof that, there exists an $n_0$ such that $Q_n ^{n_0} (x) \in (0, \tilde{x}_r)$ for all $x \in (\tilde{x}_r, \infty)$. Thus, $\lim_{m \to \infty} Q_n ^{2m}(x)=0$ for all $x >0. $ Since $Q_n$ is odd, this is also true for all $x \in \mathbb{R}$. 
		\end{enumerate}	
		%%%%%%%%%%%%%%%%%%%Graphs of Q_n%%%%%%%%%%%%%%%%
\begin{figure}[h!]
	\begin{subfigure}{.5\textwidth}
		\centering
		\includegraphics[width=1\linewidth]{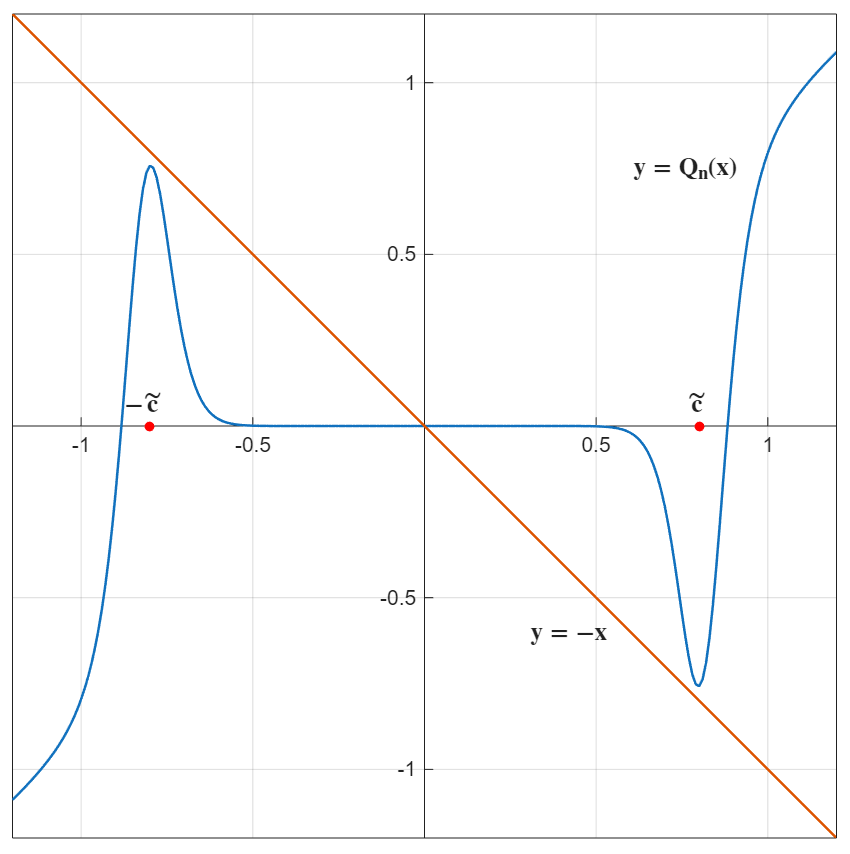}
		\caption{$n=16$}
	\end{subfigure}
	\begin{subfigure}{.5\textwidth}
		\centering
		\includegraphics[width=1\linewidth]{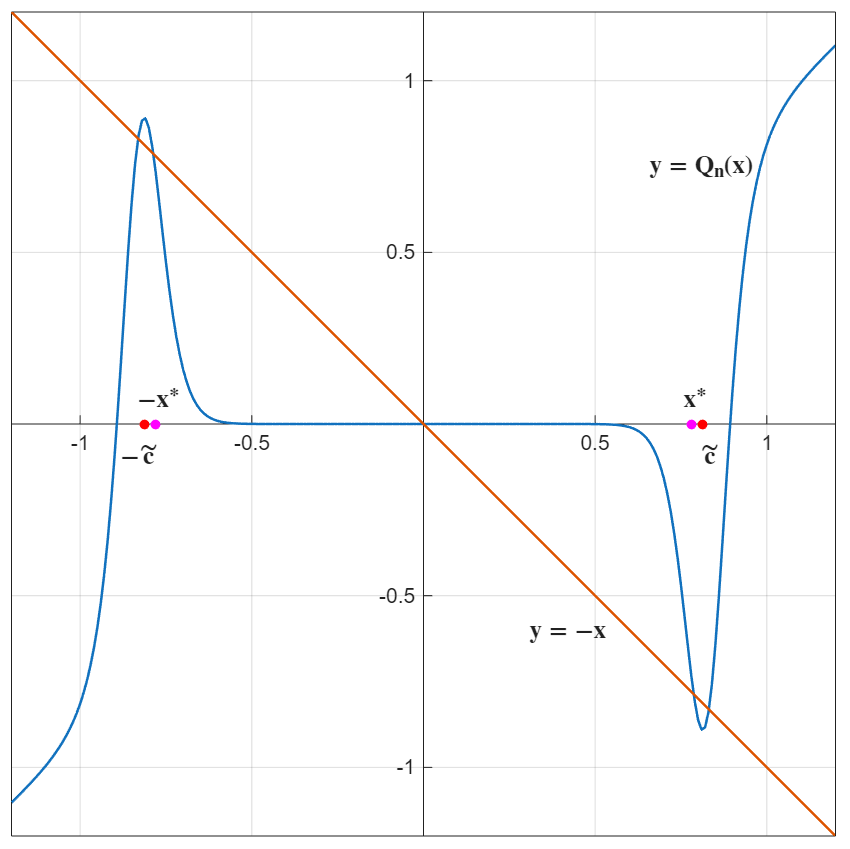}
		\caption{$n=18$}
	\end{subfigure}
	\caption{Graphs of $Q_n: \mathbb{R} \to \mathbb{R}$.}
	\label{Graph_Q}
\end{figure}	
\end{proof}

\begin{rem}\label{Crit_A_0}
   Since the set of all critical points and the immediate basin $\mathcal{A}_0$ are preserved by $n$-th order rotation about the origin, all free critical points are in $\mathcal{A}_0$ whenever $n \leq 16$.
\end{rem}

\section{Dynamics of $C_n$}\label{S4_Dynamics}
For further use, we need the following definitions. For a rational function $R$ of degree at least $2$, let $C_R$ denote the set of all critical points. The \emph{postcritical set} of $R$ is defined as $\mathcal{P}_R = \bigcup_{c \in C_R} \{R^k(c) : k \in \mathbb{N}\}$. The rational function $R$ is said to be a \emph{hyperbolic map} if $\mathcal{J}(R) \cap \overline{\mathcal{P}_R}= \emptyset$. If $\mathcal{J}(R) \cap \overline{\mathcal{P}_R}$ is a non-empty finite set, then we say $R$ is 
%\emph{sub-hyperbolic} or 
\emph{geometrically finite}.

A Fatou component is called a \emph{rotation domain} if it is either a Herman ring or a Siegel disk. The non-existence of such a domain can be proved using the invariance of critical lines (see Lemma~\ref{Inv_lines_ref-sym}).
\begin{prop}\label{no_rot-dom}
For each $n>1$,	the map $C_n$ does not have any rotation domain.
\end{prop}

\begin{proof}
If there exists a rotation domain $D$ for $C_n$, then its boundary is contained in the closure of the forward orbits of the critical points by Lemma~\ref{cpoint}. Since the critical lines  are invariant under $C_{n}$ by Lemma~\ref{Inv_lines_ref-sym},  the boundary of $D$ must be contained in the union of the critical lines. Consequently, the boundary of $D$ must contain the origin. However, $0$ is a superattracting fixed point of $C_n$ and hence belongs to the Fatou set, leading to a contradiction. This completes the proof.
\end{proof}
	
\subsection{Proofs of Theorems~\ref{Imm_unbd} and~\ref{Connected_J.set}}

We denote the immediate basins of $0$ and $1$ by $\mathcal{A}_0$ and $\mathcal{A}_1$, respectively. The other immediate basins, corresponding to the roots of $z^n=1$, are denoted by $\mathcal{A}_k$, for $k=1,2,\dots, n$, respectively, where $\mathcal{A}_{k}=\sigma^{k-1}(\mathcal{A}_1)$ and $\sigma(z)=e^{\frac{2\pi i}{n}}z $.
The following lemma will be useful in the proof of Theorem~\ref{Imm_unbd}.
	
	\begin{lem}\label{Non-A_1}
	The immediate basin	$\mathcal{A}_1$ cannot contain any negative real number.
	\end{lem}	
	
	\begin{proof}
		Suppose, on the contrary, that $\mathcal{A}_1$ contains a negative real number, say $a^{-}$. Since $\mathcal{A}_1$ is connected, there exists a simple arc $\gamma$ in $\mathbb{C} \cap \mathcal{A}_1$ joining $1$ and $a^{-}$. Because the origin belongs to  $\mathcal{A}_0$ and $\mathcal{A}_1 \cap \mathcal{A}_0 = \emptyset$, the arc $\gamma$ cannot lie entirely on the real line. Moreover, since $\mathcal{A}_1$ is symmetric with respect to the real line (a consequence of Remark~\ref{symm-R}), the reflected arc $\bar{\gamma} := \{\overline{z}: z\in \gamma\}$ is also contained in $\mathcal{A}_1$. It follows that $\Gamma = \gamma \cup \bar{\gamma}$ forms a closed curve enclosing the origin. Since $\mathcal{A}_2 = \sigma(\mathcal{A}_1)$, where $\sigma(z) = e^{\frac{2\pi  i}{n}}z$, the rotated curve $\sigma(\Gamma)$ lies entirely in $\mathcal{A}_2$ and also encloses the origin. Therefore $\Gamma \cap \sigma(\Gamma) \neq \emptyset$, which implies that $\mathcal{A}_1 \cap \mathcal{A}_2 \neq \emptyset$ (see Figure~\ref{loops_Intersect}). This is a  contradiction,  thereby completing the proof.
	%%%%%%%%%%%%%%%%%%%%%%%%%%%%%%%
	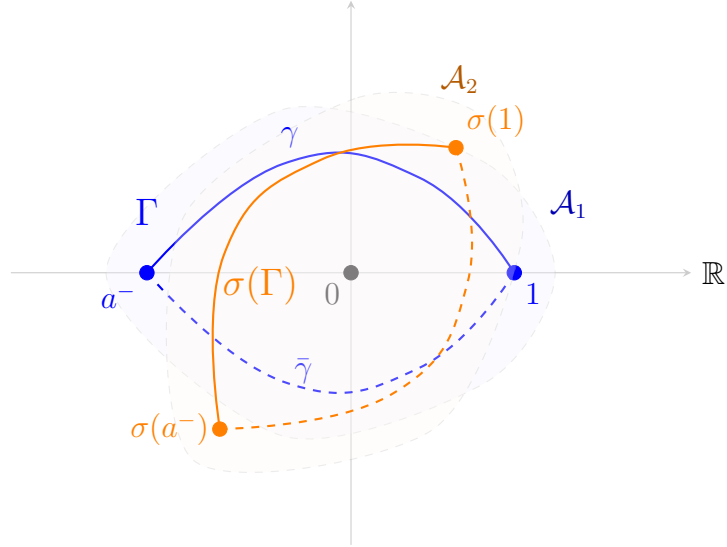
\begin{figure}[htbp]
		\centering
		\begin{tikzpicture}[scale=1.8, >=stealth]
			
			% Grid/Axis lines for reference (faint grey)
			\draw[->, help lines, color=gray!40, thin] (-2.5,0) -- (2.5,0) node[right, black] {$\mathbb{R}$};
			\draw[->, help lines, color=gray!40, thin] (0,-2) -- (0,2);
			
			% Origin point
			\filldraw[black] (0,0) circle (1.5pt) node[below left] {$0$};
			
			% --- SECTION 1: A_1 and \Gamma ---
			% Domain A_1 boundary (Light blue, symmetric background hint)
			\draw[lightgray, dashed, fill=blue!5, opacity=0.3] plot[smooth cycle] coordinates {
				(1.5,0) (1,0.8) (-0.5,1.2) (-1.8,0) (-0.5,-1.2) (1,-0.8)
			};
			\node[blue!70!black] at (1.6, 0.5) {$\mathcal{A}_1$};
			
			% Point 1 and a^-
			\filldraw[blue] (1.2, 0) circle (1.5pt) node[below right] {$1$};
			\filldraw[blue] (-1.5, 0) circle (1.5pt) node[below left] {$a^-$};
			
			% Upper arc \gamma (passing above the origin)
			\draw[thick, blue] plot[smooth, tension=0.8] coordinates {
				(1.2,0) (0.5,0.7) (-0.5,0.8) (-1.5,0)
			};
			\node[blue, above] at (-0.45, 0.84) {$\gamma$};
			
			% Lower arc \bar{\gamma} (symmetric reflection)
			\draw[thick, blue, dashed] plot[smooth, tension=0.8] coordinates {
				(1.2,0) (0.5,-0.7) (-0.5,-0.8) (-1.5,0)
			};
			\node[blue, below] at (-0.35, -0.55) {$\bar{\gamma}$};
			
			% Label for \Gamma
			\node[blue, font=\large] at (-1.5, 0.45) {$\Gamma$};

			% --- SECTION 2: A_2 and \sigma(\Gamma) ---
			% Rotation angle line for visual aid (e^{i\pi/n} where n is roughly 3 or 4 here)
			% Let's rotate by 50 degrees
			\begin{scope}[rotate=50]
				% Domain A_2 boundary hint
				\draw[lightgray, dashed, fill=orange!5, opacity=0.3] plot[smooth cycle] coordinates {
					(1.5,0) (1,0.8) (-0.5,1.2) (-1.8,0) (-0.5,-1.2) (1,-0.8)
				};
				\node[orange!70!black] at (1.6, 0.3) {$\mathcal{A}_2$};
				
				% Rotated upper arc
				\draw[thick, orange] plot[smooth, tension=0.8] coordinates {
					(1.2,0) (0.5,0.7) (-0.5,0.8) (-1.5,0)
				};
				
				% Rotated lower arc
				\draw[thick, orange, dashed] plot[smooth, tension=0.8] coordinates {
					(1.2,0) (0.5,-0.7) (-0.5,-0.8) (-1.5,0)
				};
				
				% NEW: Orange dots at the endpoints inside the rotated scope
				\filldraw[orange] (1.2, 0) circle (1.5pt) node[above right] {$\sigma(1)$};
				\filldraw[orange] (-1.5, 0) circle (1.5pt) node[left] {$\sigma(a^-)$};
				
				% Label for \sigma(\Gamma)
				\node[orange, font=\large] at (-0.5, 0.45) {$\sigma(\Gamma)$};
			\end{scope}
			
			% % --- INTERSECTION HIGHLIGHT ---
			% Mark one of the clean intersection points where the loops cross to emphasize \Gamma \cap %\sigma(\Gamma) \neq \emptyset
			% \filldraw[red] (-0.32, 0.68) circle (1.8pt);
			% \filldraw[red] (0.32, -0.68) circle (1.8pt);
			% \draw[red, <-] (-0.28, 0.72) .. controls (-0.1, 1.2) and (0.3, 1.2) .. (0.6, 1.3)
			% node[right, black, font=\small] {Intersection $\Gamma \cap \sigma(\Gamma) \neq \emptyset$};
			
		\end{tikzpicture}
		\caption{Illustration of the contradiction arising from the intersection of $\Gamma$ and $\sigma(\Gamma)$.}
		\label{loops_Intersect}
	\end{figure}
\end{proof}	
	
 We now establish the simple connectivity and unboundedness of the immediate basins corresponding to the non-zero roots of $p_n$ by proving Theorem~\ref{Imm_unbd}. 
	
	\begin{proof}[Proof of Theorem~\ref{Imm_unbd}]
		Since each immediate basin $\mathcal{A}_i$, $i=1,2,\dots, n$, is the image of $\mathcal{A}_1$ under the map $z \mapsto \lambda z$ where ${\lambda}^n=1$ (see Lemma~\ref{rot-symm}), it is sufficient to prove this theorem for $\mathcal{A}_1$.
		\par It follows from Proposition~\ref{dynamics-criticallines} that $(e_2, \infty) \subset \mathcal{A}_1$. In particular, the immediate basin $\mathcal{A}_1$ is unbounded. 
		
		\par In order to show that $\mathcal{A}_1$ is simply connected, we first prove that $\mathcal{A}_1$ does not contain any free critical point. If $n$ is even, this is immediate from Remark~\ref{No_free_cr_for_even}. Suppose now that $n$ is odd. Then $C_n$ has a real free critical point and that is negative. It follows from Lemma~\ref{Non-A_1} that the free critical point is not in $\mathcal{A}_1$. Hence, $\mathcal{A}_1$ contains exactly one (distinct) critical point. \cite[Theorem~9.3]{Milnor_book} gives that the immediate basin of a superattracting fixed point is simply connected whenever it does not contain any critical point other than the superattracting fixed point itself. Therefore, $\mathcal{A}_1$ is simply connected.
	\end{proof}
 We now prove the connectedness of the Julia set.
	\begin{proof}[Proof of Theorem~\ref{Connected_J.set}]
		It follows from the proof of Theorem~\ref{Imm_unbd} that for each $i=1,2,\dots, n$, the immediate basin $\mathcal{A}_i$ is unbounded. Therefore, its boundary $\partial \mathcal{A}_i$ contains a pole by the first part of Lemma~\ref{Jconnected}, as $\infty$ is a repelling fixed point of $C_n$. Moreover, each $\mathcal{A}_i$ is simply connected. Thus, the unbounded Julia component contains all the poles. This is a consequence of the fact that the set of all poles as well as the set of all $\mathcal{A}_i$s are preserved by rotations of order $n$ about the origin (see Lemma~\ref{Sym_C_p}).  It then follows by the second part of Lemma~\ref{Jconnected} that the Julia set of $C_n$ is connected.
	\end{proof}
%%%--------------------
%%%-------------------
\subsection{Proofs of Theorems~\ref{Bdd_imm} and~\ref{Convergent}}
 Let $U$ be a Fatou component, and $z_0\in U$. A boundary point $z^*$ is called accessible from $z_0$ if there is a simple curve $\gamma:[0,1)\to U$ such that $\gamma(0)=z_0$ and $\gamma(1)=z^*$. The homotopy classes of all such curves are called accesses to $z^*$.

We next study the accessibility of the extraneous fixed points from the immediate basin $\mathcal{A}_0$. The following proposition shows that these points lie on $\partial \mathcal{A}_0$ and are accessible through a unique access.
\begin{prop}\label{e_1onBdryA_0}
	The extraneous fixed points, which are solutions of $z^{n}=e_{1}^{n}$, lie on $\partial \mathcal{A}_{0}$. Moreover, there is exactly one access to these fixed points from $\mathcal{A}_0$.
\end{prop}
\begin{proof}
	Consider the interval $(0,e_{1})$. On this interval, $C_{n}$ is increasing, does not have any critical point, fixed point or pole, and satisfies $C_{n}(x) < x$. Therefore,  
	$$\lim_{k\rightarrow\infty} C_{n}^{k}(x) = 0 \quad \text{for all}\quad x \in (0,e_{1}).$$
	This implies that $\mathcal{A}_{0}$ contains $(0,e_{1})$, and hence, by symmetry, all extraneous fixed points satisfying $z^{n}=e_{1}^{n}$ lie on $\partial \mathcal{A}_{0}$.  
	
	Now we prove the second part. Let $\deg_{\mathcal{A}_{0}} C_{n} = d$. Since  $\mathcal{A}_0$ is invariant under the rotations of order $n$ about the origin, $d$ is either $n+1$ (if $\mathcal{A}_{0}$ contains no free critical points) or $2n+1$ (if $\mathcal{A}_{0}$ contains free critical points).  Since $\mathcal{A}_{0}$ is simply connected, there exists a normalized Riemann map $\varphi : \mathbb{D} \rightarrow \mathcal{A}_{0}$ such that $\varphi(0)=0$, and $\varphi((0,1)) \subseteq \mathbb{R}^{+} \cap \mathcal{A}_{0}$. Consider the inner function    
	$$f = \varphi^{-1} \circ C_{n} \circ \varphi : \mathbb{D} \rightarrow \mathbb{D}$$ associated to $C_{n}|_{\mathcal{A}_{0}}$. It follows that $\deg_{\mathbb{D}} f = d$. Using Schwarz's reflection principle, $f$ can be extended to a rational map $F : \widehat{\mathbb{C}} \rightarrow \widehat{\mathbb{C}}$ of degree $d$.
	
	Consequently, $F$ has $d+1$ fixed points. By construction, $0$ and $\infty$ are superattracting fixed points of $F$, having immediate basins $\mathbb{D}$ and $\widehat{\mathbb{C}} \setminus \overline{\mathbb{D}}$, respectively.  The remaining $d-1$ fixed points of $F$ lie on the unit circle. Let these be $\beta_{1}, \beta_{2}, \dots, \beta_{d-1}$. Each of these fixed points characterizes a distinct access to a fixed point on $\partial \mathcal{A}_{0}$ (as a fixed landing ray).  
	
	We claim that each fixed point from the set
	$\{z : z^{n} = e_{1}^{n}\}$ corresponds to exactly one such $\beta_{i}$ for $i = 1, 2, \dots, d$.  Consider $e_1$. Since $(0,e_1) \subseteq \mathcal{A}_0$, this interval determines an access to $e_{1}$ from $\mathcal{A}_{0}$, which corresponds to, say $\beta_{1}$. Suppose there exists a second access to $e_{1}$ from $\mathcal{A}_0$. This access must correspond to another fixed point $\beta_{i}$, without loss of generality, say $\beta_{2}$. Therefore, this second access must lie in either the upper or lower half-plane.  
	
	Symmetry with respect to the real axis implies that there must be another access to $e_{1}$ from $\mathcal{A}_0$ lying in the opposite half- plane, corresponding to some $\beta_{3}$ (say). Thus, there would be at least $3$ accesses to $e_{1}$ from $\mathcal{A}_{0}$.  This suggests that there are at least $3n$ total accesses to these extraneous fixed points. Hence, we would have $d-1 \ge 3n$, which implies $d \geq 3n+1$. However, this is not possible since the maximum value of $d$ is $2n+1$ for $(n \geq 2)$. This completes the proof.
\end{proof} 

\begin{rem}\label{one_access-A1}
	Using the arguments from the proof of Proposition~\ref{e_1onBdryA_0}, it follows that each immediate basin corresponding to a non-zero root of $p_n$ has exactly one access to $\infty$.
\end{rem}	

From Remark~\ref{Crit_A_0}, it follows that all free critical points are in $\mathcal{A}_0$. In this case, if $c$ is a free critical point, then $|C_n(c)|<|c|$ (see Lemma~\ref{c_and_c^*}). The next result deals with the case when $|C_n(c)|>|c|$.

\begin{prop}\label{n-geq-17}
	If  $c$ is a free critical point of $C_n$ and $|C_n(c)|>|c|$, then $c \notin \mathcal{A}_0$. 
\end{prop}
\begin{proof}
	Since the set of all free critical points, as well as the immediate basin $\mathcal{A}_0$ of $0$, are preserved under  $z \mapsto \lambda z$ for $\lambda^n =1$, it is enough to prove that any one free critical point is not in $ \mathcal{A}_0$.
	There are two cases, depending on the parity of $n$.  
	\par
	\noindent \underline{\textit{Case 1: $n$ is odd.}}	\\
	Let $n$ be odd. Then $C_n$ has a real free critical point  $c_r$ (see Figure~\ref{L_k}(a)). The corresponding critical value $c_r^* =C_n (c_r)$ satisfies $c_r^* >e_1$ as $n \geq 17$ (see Lemma~\ref{c_and_c^*} and Proposition~\ref{Odd_free_cr}). Suppose, on the contrary, that $c_r \in \mathcal{A}_0$. Since $\mathcal{A}_0$ is forward invariant under $C_n$, the point $c_r^*$ must also belong to $\mathcal{A}_0$. Since $\mathcal{A}_0$ is connected, there must exist a simple curve $\gamma$ in $\mathcal{A}_0$ joining $0$ and $c_r^*$. This curve cannot lie entirely on the real line as $\infty$ and $e_1$ belong to the Julia set. By the symmetry of $\mathcal{A}_0$ about the real axis, the complex conjugate curve $\bar{\gamma}=\{\overline{z}: z \in \gamma\}$ is also contained in $\mathcal{A}_0$ and connects $0$ to $c_r^*$ and $\Gamma =\gamma \cup \bar{\gamma}$ forms or contains a simple closed curve in $\mathcal{A}_0$ enclosing $e_1$. Therefore, $\Gamma$ separates the component of the Julia set containing $e_1$ from the one containing $\infty$  (see Figure~\ref{isolated_Julia_comp} for a schematic demonstration). This contradicts the connectedness of the Julia set of $C_n$. Therefore $c_r \notin \mathcal{A}_0$.

	\noindent	\underline{\textit{Case 2: $n$ is even.}}\\
	Now consider even $n$. Then $C_n$ has a free critical point $c_0$ on $L_1$ (see Figure~\ref{L_k}(b)).
	It follows from Lemma~\ref{Inv_lines_ref-sym}(2) that $C_{n} (\lambda_1 x)=\lambda_1 Q_n (x)$ for all $x \in \mathbb{R}$.
	Following Equation~(\ref{Q-deri}), we have
	$$Q_n'(x) = \frac{n(n+1)^2(2n+1)x^{n}(x^{n}+1)^{2} \big(x^{n} - \tilde{c}^n\big)}{2\big((n+1)x^{n}+1\big)^4},$$
	where $\tilde{c}$ denotes the positive critical point of $Q_n$. Thus the function  $Q_n$ decreases in $(0, \tilde{c})$ and increases thereafter.
	\par  
	Each free critical point $c$ of $C_n$ is of the form $\lambda c_0$ for some $\lambda$ with $\lambda^n =1$. Since $c_0$, the free critical point of $C_n$ lying on $L_1$ is nothing but $\lambda_1 \tilde{c}$, we have $c= \lambda \lambda_1 \tilde{c}$. It follows from the hypothesis of this proposition and Lemma~\ref{Inv_lines_ref-sym} (2) that $|Q_n (\tilde{c})|=|C_n (c)|> |c| =|\tilde{c}|$. Since $\tilde{c}>0$ and $Q_n (\tilde{c}) <0$, we have $Q_n (\tilde{c}) +\tilde{c}<0$.  On the other hand, since the origin is a real superattracting fixed point of $Q_n$, $Q_n (x)+x>0$ for sufficiently small $x>0$.By the Intermediate Value Theorem, there exists a unique point $x^*$ (choose smallest such $x^*$) in $(0, \tilde{c})$ such that $Q_n(x^*) = -x^*$.  Moreover, since $Q_n$ is odd, $Q_n^2(x^*) = x^*$, and therefore  
	$\{-x^*, x^*\}$ forms a $2$-periodic cycle. Note that $\lim_{x \to +\infty}Q_n (x)=+\infty.$ The graph of $Q_{18}$ is given in Figure~\ref{Graph_Q}(b) to demonstrate these properties.
	\par  
	The function $Q_n ^2: (0, x^*) \to (0, x^*)$ is increasing and satisfies $Q_n ^2 (x)<x$. It follows that $\lim_{k \to \infty}Q_n ^{2k} (x)=0$ for all $0< x < x^*$. Clearly, the line segment $l=\{\lambda_1 x: -x^* < x< x^*\}$ contained in $L_1$ is exactly equal to $L_1 \cap \mathcal{A}_0$. If $c_0 =\lambda_1 \tilde{c}$ belongs to $\mathcal{A}_0$ then there exists a simple closed curve in $\mathcal{A}_0$ joining $c_0$ and $C_n (c_0)$ that separates $\lambda_1 x^*$ from $\infty$. Both of these points are in the Julia set. This contradicts the simple connectivity of $\mathcal{A}_0$. The argument is identical to that of Case~1. Thus, the free critical point $c_0$ does not belong to $\mathcal{A}_{0}$. 
	\begin{figure}[htbp]
		\centering
		\begin{tikzpicture}[scale=1.5]
		
		% Draw the real axis line (olive/yellow-green color)
		\draw[->, thick, olive] (-1.5,0) -- (4.5,0);
		
		% Draw the outer blue curve (simulating the hand-drawn cloud shape)
		\draw[thick, blue] plot[smooth cycle, tension=0.8] coordinates {
			(-0.31,0) (0,0.5) (0.8,0.7) (1.5,0.6) (2.5,0.8) (3.5,0.5) (3.9,0) 
			(3.4,-0.6) (2.4,-0.9) (1.5,-0.7) (0.7,-0.8) (0,-0.5)
		};
		
		% Draw the inner isolated Julia component (black cloud shape)
		\draw[thick, black, fill=white] plot[smooth cycle, tension=0.7] coordinates {
			(0.9,0) (1.2,0.2) (1.7,0.1) (2.2,0.2) (2.5,0) 
			(2.1,-0.2) (1.6,-0.1) (1.2,-0.2)
		};
		
		% Label the isolated Julia component with an arrow
		\draw[thick, <-, >=stealth] (1.7, 0.25) .. controls (1.9, 0.6) and (2.1, 0.7) .. (3.2, 0.8) 
		node[right] {Isolated Julia component};
		
		% --- ADDED LINE: Draw the completed line from e_1 to c* ---
		\draw[ thick, olive] (0.88,0) -- (3.88,0);
		
		% Draw and label the specific points on the axis
		% Point 0
		\filldraw[olive] (-0.3,0) circle (1.5pt);
		\node[below left, olive] at (-0.3,0) {$0$};
		
		% Point e_1
		\filldraw[olive] (0.88,0) circle (1.5pt);
		\node[below left, olive] at (0.88,0) {$e_1$};
		
		% Point c*
		\filldraw[olive] (3.88,0) circle (1.5pt);
		\node[below right, olive] at (3.88,0) {$c^*$};
		
		\end{tikzpicture}
		\caption{Representation of the isolated Julia component.}
		\label{isolated_Julia_comp}
	\end{figure}
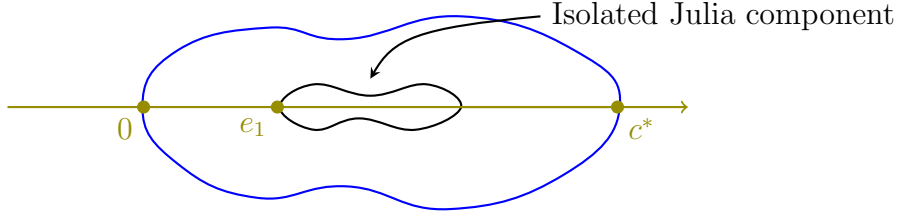
\end{proof}
%---------------------------------

%%%%%%%%%%%%%%%%%%%%%%%%%%%%%%%%%%%%
%\begin{figure}[h!]
%	\begin{subfigure}{.5\textwidth}
%		\centering
%		\includegraphics[width=1\linewidth]{Plot_Q_n16.png}
%		\caption{$n=16$}
%	\end{subfigure}
%	\begin{subfigure}{.5\textwidth}
%		\centering
%		\includegraphics[width=1\linewidth]{Plot_Q_n18.png}
%		\caption{$n=18$}
%	\end{subfigure}
%	\caption{The graphs of $Q: \mathbb{R} \to \mathbb{R}$.}
%	\label{Graph_Q}
%\end{figure}
%%%%%%%%%%%%%%%%%%%%%%%%%%%%%%%%%%%%

 With these results, we are now in a position to prove Theorem~\ref{Bdd_imm}, which completely characterizes the boundedness and unboundedness of $\mathcal{A}_{0}$ according to the value of $n$.

\begin{proof}[Proof of Theorem~\ref{Bdd_imm}]
	The proof is carried out by considering the following two cases:
	
	\noindent
	Case~(1): $n\leq 16$. In this case, $|c^*|<|c|$.
	
	We prove that $\mathcal{A}_0$ is unbounded. For odd $n$, the immediate basin $\mathcal{A}_0$ corresponding to $0$ contains the unbounded interval $(-\infty,e_1)$ by Proposition~\ref{dynamics-criticallines}(1).
    For even $n$,  it follows from Remark~\ref{dynamics-invariantlines}(2(a)) that $\lambda_k ^{-1}C_n (\lambda_k x) =Q_n (x)$ for all odd $k$ and real $x$. Since $ \lim_{m \to \infty} Q_n ^m (x)= 0$ for all 
    $x \in \mathbb{R}$ (see Proposition~\ref{dynamics-criticallines}(2)), the line $L_k = \{\lambda_k x: x \in \mathbb{R}\}$ for odd $k$ is contained in $\mathcal{A}_0$. Therefore, $\mathcal{A}_0$ is unbounded. \\

	\noindent
	Case~(2): $n\geq 17$. In this case, $|c^*|>|c|$.
	
	We prove that $\mathcal{A}_0$ is bounded. It follows from Proposition~\ref{n-geq-17} that $\mathcal{A}_0$ does not contain any free critical point. Thus, $\mathcal{A}_0$ contains exactly one distinct critical point, namely, the origin. It follows that $\deg_{\mathcal{A}_0} C_{n} = n+1$. Consequently, the inverse Böttcher coordinate can be extended globally to $\mathcal{A}_0$, i.e., $C_{n}$ restricted to $\mathcal{A}_0$ is conjugate to the monomial $z \mapsto z^{n+1}$. By Proposition~\ref{e_1onBdryA_0}, there are exactly $n$ internal fixed rays, each landing at one of the $n$ finite extraneous fixed points of $C_n$.  Consequently, no internal fixed ray lands on $\infty$. As $\infty$ is a repelling fixed point with positive multiplier, if it is a boundary point of $\mathcal{A}_0$, then it would necessarily be the landing point of a periodic ray. Therefore, the period of such a ray would have to be at least $2$. However, near $\infty$, $C_{n}$ is locally conjugate to a linear map of the form $z \mapsto \lambda z$.  This local linearity prevents the existence of such periodic rays accumulating at $\infty$, implying that $\infty \notin \partial \mathcal{A}_0$. Hence $\mathcal{A}_0$ is bounded.
\end{proof}

 Figures~\ref{n-dynamics}(a), (b), and (c) illustrate the unboundedness of $\mathcal{A}_0$ for $n=5,6$, and $16$, respectively, whereas, Figure~\ref{n-dynamics}(d) illustrates the boundedness of $\mathcal{A}_0$ for $n=17$.

 The unboundedness of $\mathcal{A}_0$ and Proposition~\ref{Odd_free_cr} set the stage for proving that $C_n$ is convergent whenever $n\leq16$ or $n$ is odd, which is precisely the statement of Theorem~\ref{Convergent}.

\begin{proof}[Proof of Theorem~\ref{Convergent}]
For $n \leq 16$, the proof is done as all free critical points are in $\mathcal{A}_0$ (see Remark~\ref{Crit_A_0}). Now consider $n \geq 17$ and $n$ odd. It follows from Proposition~\ref{Odd_free_cr} that the critical value corresponding to the real free critical point, or its image is greater than $e_2$, the larger real extraneous fixed point of $C_n$. Since $(e_2, \infty)$ is contained in the immediate basin of $1$, it follows that the real free critical point  is in the basin of $1$. As the set of all free critical points  and the set  $\cup_{i=1}^{n} \mathcal{A}_i$ are preserved under rotations (about the origin) of order $n$, it is clear that each free critical point is in the immediate basin of some non-zero root of $p_n$.  Hence, the map $C_n$ is convergent.
\end{proof}

Recall that a rational function is called geometrically finite if every critical point belonging to the Julia set has a finite forward orbit. Tan and Yin established in \cite[Theorem A]{Tan1996} that, for a geometrically finite rational map $R$, the Julia set $\mathcal{J}(R)$ is locally connected whenever it is connected. As the free critical points of $C_n$ are in the Fatou set, it is a geometrically finite rational map, and from Theorem~\ref{Connected_J.set} we conclude that $\mathcal{J}(C_n)$ is locally connected whenever $n\leq 16$ or odd.

\subsection{Proof of Theorem~\ref{Equal_sym}}

	It follows from Theorem~\ref{Convergent} that $C_n$ does not have any periodic (of period at least $2$) attracting or parabolic periodic domain whenever $n \leq 16$ or $n$ is odd. However, for $n\geq 18$ even, the existence of such a domain is not disproved. If such a domain exists, then the free critical points should belong to the union of such domains. In such a scenario, although $C_n$ is not convergent, it will be a geometrically finite map. Hence, the Julia set $\mathcal{J}(C_n)$ will be locally connected.
	
	Moreover, for $n$ even, each of the lines $L_{2j+1}$, for $j=0, \dots, (n-2)/2$, contains exactly two free critical points, both of which have the same modulus. As these lines are invariant and pass through the origin, a parabolic domain $\Omega$ of $C_n$ can contain at most one free critical point. Therefore, each of $\sigma^{l}(\Omega)$, where $\sigma^{l}(z)=e^{\frac{2\pi l}{n}}z$ for $l=0,1,\dots, n-1$, contains at most one free critical point.
	
	The next result will be helpful to prove Theorem~\ref{Equal_sym}.
	\begin{prop}\label{Para_access}
		Let $\Omega$ be an unbounded parabolic domain that contains a free critical point. Then there are at least two accesses to infinity from $\Omega$.
	\end{prop}
	
	\begin{proof}
		%It follows from Theorem~\ref{Convergent} that $n\geq 18$, and therefore, $|c^*|>|c|$. 
		As the dynamics of $C_n$ and $Q$ are equivalent, we assume that $Q$ has an unbounded parabolic domain $\Omega_Q$ that contains a free critical point. Moreover, without loss of generality, assume that $c_0\in \Omega_Q$, $c_0$ is the real free critical point of $Q$.
		Since the Julia set of $C_n$, and therefore of $Q$, is locally connected, $\infty$ is accessible from $\Omega_Q$. 
		Note that each of the intervals $(-\infty, c_0)$ and $(c_0,\infty)$ contains a root of $Q(x)=0$. Thus, neither of those intervals can lie entirely on $\Omega_Q$. Therefore, an access to $\infty$ from $\Omega_Q$ should be either on the upper or lower half-plane. The symmetry of $Q$ about the real axis guarantees the existence of another access to $\infty$ from $\Omega_Q$. 
	\end{proof}
	
	\begin{rem}
		Let $V$ be an unbounded preimage component corresponding to $\Omega$, i.e.,
		$C_n(V)=\Omega.$ Then $V$ contains no critical point of $C_n$, and therefore, $V$ also has at least two accesses to $\infty$.
		
		Indeed, in this case the restriction
		$C_n:V\to \Omega$
		is conformal. Hence distinct accesses to $\infty$ in $\Omega$ lift homeomorphically to distinct accesses in $V$. Since there is no critical branching in $V$, two distinct lifted accesses cannot merge together. Therefore, $V$ must also have at least two accesses to $\infty$.
	\end{rem}

We are now in a position to prove Theorem~\ref{Equal_sym}, where equality of $\Sigma p_n$  and $\Sigma C_n$ is established.

\begin{proof}[Proof of Theorem~\ref{Equal_sym}]
	Since $\mathcal{A}_1$ is unbounded, for every $\sigma\in \Sigma C_n$, $\sigma(\mathcal{A}_1)$ is an unbounded Fatou component.
	For $n\leq 16$ or $n$ odd, $C_n$ is convergent. It follows from Lemma~\ref{bdd_pre-im} that all pre-images of each of the immediate basins are bounded. Therefore, $\Sigma C_n$ contains at most $n$ elements, and hence, this set is identical with $\Sigma p$.
	
	Consider $n\geq 18$ and even. 
	
	It follows from Proposition~\ref{no_rot-dom} that for every $n$, $C_n$ does not have any rotation domain. If $C_n$ has no parabolic domain, then by \cite[Theorem C]{Nayak-Pal2025}, the result holds.
	
	If there are parabolic domains, but all of them are bounded, then we are also done following Lemma~\ref{bdd_pre-im}. 
	
	Let $\sigma(\mathcal{A}_1)=\Omega$, where $\Omega$ is an unbounded parabolic domain. Then, for some $k\in \mathbb{N}$, $C_n^k(\Omega)$ contains a free critical point. It follows from Proposition~\ref{Para_access} that $\Omega$ has at least two accesses to $\infty$. It is not possible as $\mathcal{A}_1$ has exactly one access (see Remark~\ref{one_access-A1}). Thus, $\mathcal{A}_1$ can only be mapped to $\mathcal{A}_i$, $i=1,2,\dots, n$, under any rotation in $\Sigma C_n$. Hence, the symmetry groups of $p$ and $C_n$ are identical. 
\end{proof}
%%%-------
\begin{figure}[h!]
	\begin{subfigure}{.5\textwidth}
		\centering
		\includegraphics[width=1.1 \linewidth]{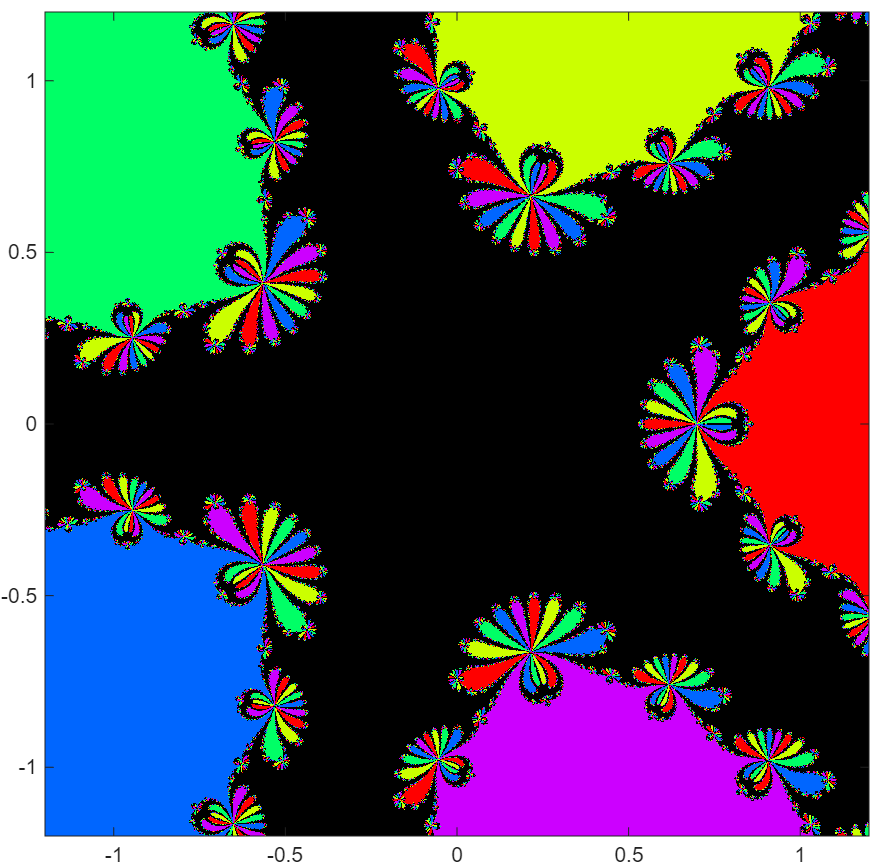}
		\caption{$n=5$}
	\end{subfigure}
	\begin{subfigure}{.5\textwidth}
		\centering
		\includegraphics[width=1.1 \linewidth]{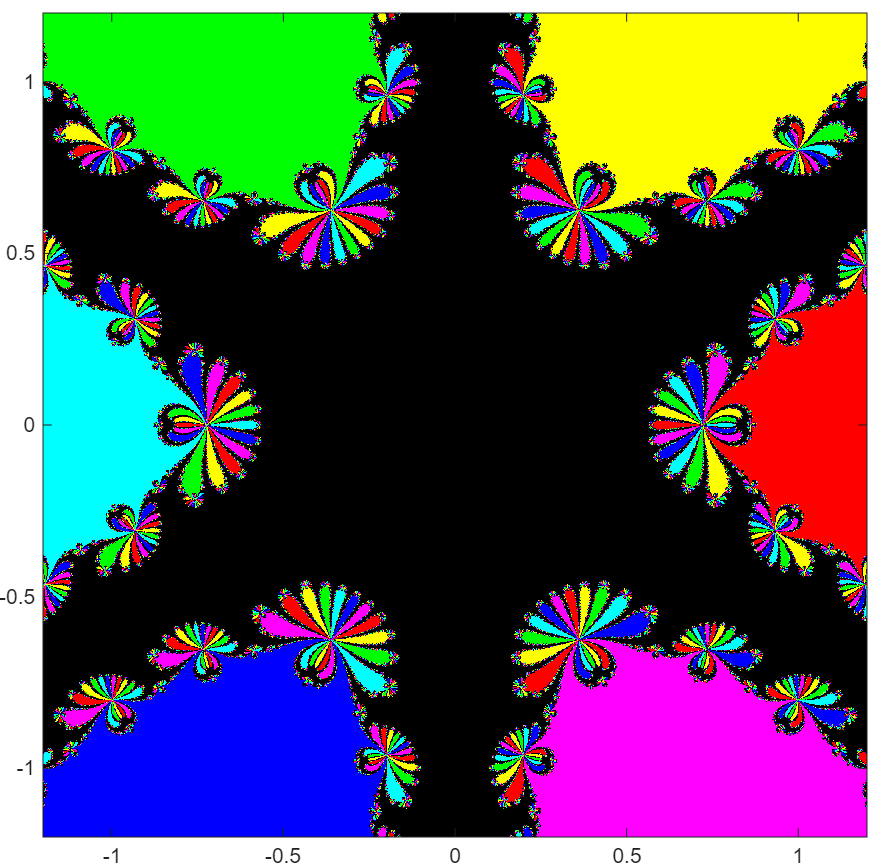}
		\caption{$n=6$}
	\end{subfigure}
	\vspace{0.5cm}
	\begin{subfigure}{.5\textwidth}
		\centering
		\includegraphics[width=1.1 \linewidth]{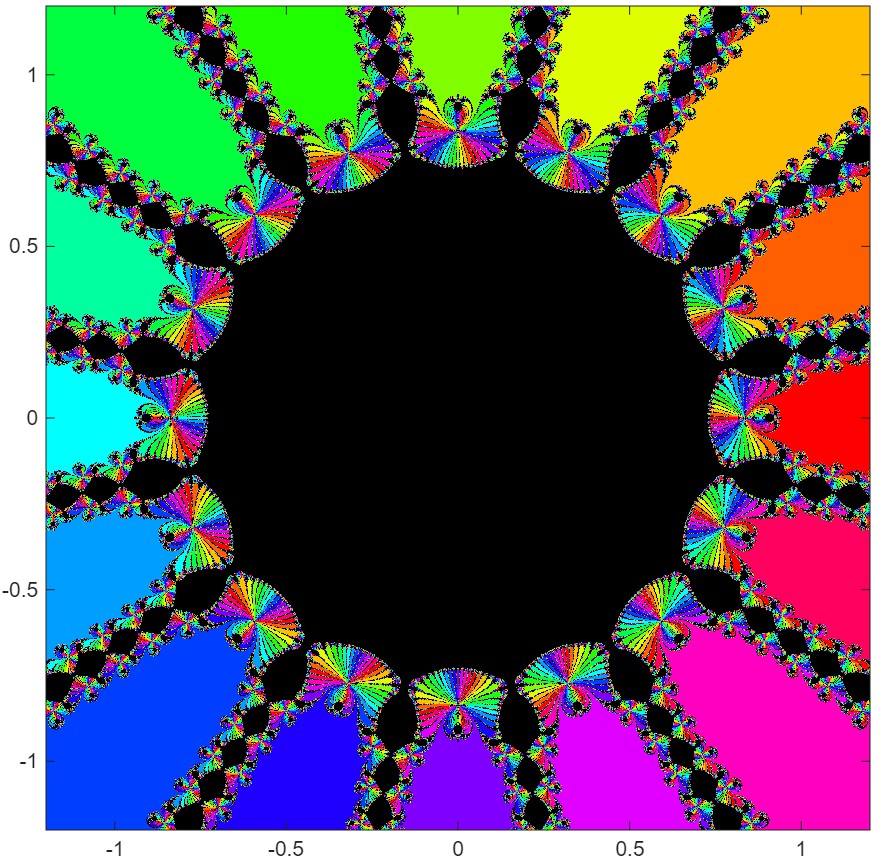}
		\caption{$n=16$}
	\end{subfigure}
	\begin{subfigure}{.5\textwidth}
		\centering
		\includegraphics[width=1.1 \linewidth]{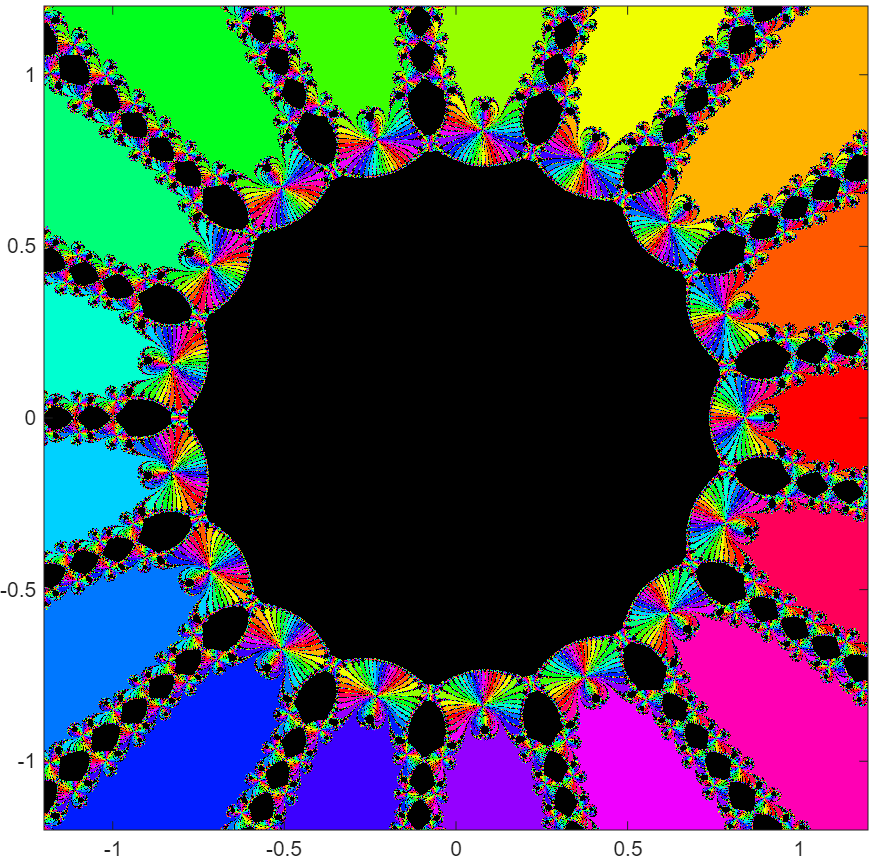}
		\caption{$n=17$}
	\end{subfigure}
	\caption{Dynamical planes of $C_n$.}
	\label{n-dynamics}
\end{figure} 
%%%------
 Figure~\ref{n-dynamics} represents the dynamical planes of $C_n$ for certain values of $n$. Different shades of colors represent different attracting basins, with black representing the basin of $0$. It can be observed that $\mathcal{A}_0$ is unbounded for $n=5$, $6$, and $16$, while it is bounded for $n=17$. It can also be seen that the Julia set of $C_n$ is preserved by $n$-th order rotations about the origin.
%%%--------------------
%%%--------------------

\section{Concluding remarks}\label{S5_Con-rem}
The convergence of $C_n$ for even $n$ is not straightforward. Numerical evidence (see Table~\ref{tab:free_critical_point_convergence}) indicates that the free critical points are in the basin of $0$ for all even $n \leq 112$. Although not included in the table, our computations in fact show that the forward orbits of all the free critical points are in $ \mathcal{A}_0$ for $52 \leq n  \leq 98$, whereas for larger values of $n$, the forward orbit of a free critical point does not enter $\mathcal{A}_0$ even after $10^{7}$  iterations. Some indicative values are presented in Table~\ref{tab:free_critical_point_convergence}.

  Based on these observations, together with theoretical insights developed in this article, we propose the following question for further investigation:
\textit{Is the map $C_n$ convergent for every even $n \geq 18$?}
\begin{table}[h!]
	\centering
	\renewcommand{\arraystretch}{1.3}
	\begin{tabular}{|c|c|c|c|c|}
		\hline
		$n$ & Free critical point $c$ & $|x^*|$ & $k$ & Is $|C_n^k(c)|<|x^*|$?\\
		\hline
		$18$  & $0.8008396 + 0.1412096 i$ & $0.7870328$ & $2$      & Yes \\
		\hline
		$20$  & $0.8162204 + 0.1292766 i$ & $0.7906919$ & $3$      & Yes \\
		\hline
		$22$  & $0.8291154 + 0.1192088 i$ & $0.7971412$ & $8$      & Yes \\
		\hline
		$24$  & $0.8401075 + 0.1106022 i$ & $0.8041752$ & $11$     & Yes \\
		\hline
		$26$  & $0.8496070 + 0.1031610 i$ & $0.8111648$ & $9$      & Yes \\
		\hline
		$28$  & $0.8579114 + 0.0966634 i$ & $0.8178806$ & $11$     & Yes \\
		\hline
		$30$  & $0.8652426 + 0.0909406 i$ & $0.8242398$ & $13$     & Yes \\
		\hline
		$32$  & $0.8717695 + 0.0858618 i$ & $0.8302206$ & $16$     & Yes \\
		\hline
		$34$  & $0.8776230 + 0.0813237 i$ & $0.8358292$ & $18$     & Yes \\
		\hline
		$36$  & $0.8829068 + 0.0772443 i$ & $0.8410839$ & $21$     & Yes \\
		\hline
		$38$  & $0.8877036 + 0.0735572 i$ & $0.8460079$ & $95$     & Yes \\
		\hline
		$40$  & $0.8920805 + 0.0702082 i$ & $0.8506257$ & $26$     & Yes \\
		\hline
		$42$  & $0.8960926 + 0.0671528 i$ & $0.8549616$ & $29$     & Yes \\
		\hline
		$44$  & $0.8997855 + 0.0643539 i$ & $0.8590383$ & $32$     & Yes \\
		\hline
		$46$  & $0.9031975 + 0.0617804 i$ & $0.8628770$ & $35$     & Yes \\
		\hline
		$48$  & $0.9063606 + 0.0594060 i$ & $0.8664970$ & $38$     & Yes \\
		\hline
		$50$  & $0.9093022 + 0.0572084 i$ & $0.8699160$ & $41$     & Yes \\
		\hline
		$100$ & $0.9476911 + 0.0297824 i$ & $0.9190445$ & $143$    & Yes \\
		\hline
		$110$ & $0.9515649 + 0.0271840 i$ & $0.9244602$ & $151$    & Yes \\
		\hline
		$112$ & $0.9522659 + 0.0267180 i$ & $0.9254492$ & $155$    & Yes \\
		\hline
		$114$ & $0.9529452 + 0.0262677 i$ & $0.9264102$ & $10^{7}$ & No \\
 		\hline
 		$116$ & $0.9536036 + 0.0258324 i$ & $0.9273442$ &  $10^{7}$ & No \\
%		\hline
%		$118$ & $0.954242 + 0.0254114 i$ & $0.9282525$ &  $10^{7}$ & No \\
		\hline
		$120$ & $0.9548622 + 0.0250039 i$ & $0.9291361$ &  $10^{7}$ & No \\
		\hline
		$150$ & $0.9623706 + 0.0201588 i$ & $0.9400130$ &  $10^{7}$ & No \\
		\hline
		$200$ & $0.9703066 + 0.0152428 i$ & $0.9518609$ &  $10^{7}$ & No \\
		\hline
		$300$ & $0.9788218 + 0.0102506 i$ & $0.9649923$ &  $10^{7}$ & No \\
		\hline
		$400$ & $0.9833816 + 0.0077236 i$ & $0.9722134$ &  $10^{7}$ & No \\
		\hline
		$500$ & $0.9862501 + 0.0061969 i$ & $0.9768289$ &  $10^{7}$ & No \\
		\hline
		$1000$ & $0.9924205 + 0.0031178 i$ & $0.9869703$ &  $10^{7}$ & No \\
		\hline
	\end{tabular}
%	\caption{For even $n$ with $n\geq 18$, the free critical point and the $2$-periodic point closest to the origin are denoted by $c$ and $x^*$, respectively, and $k$ is the number of iterations.}
	\caption{For even $n$ with $18 \leq n\leq 50$ and $n=100,110,112$, the free critical point and the $2$-periodic point closest to the origin are denoted by $c$ and $x^*$, respectively.  The first, second  and third columns contain the values of $n$, $c$ and $|x^*|$, respectively, whereas the fourth column gives the number of iterations used to determine whether  $C_n ^k (c) \in (-x^*, x^*)$ or not.  }
	\label{tab:free_critical_point_convergence}
\end{table}

Finally, we present a  comparison of Newton's method $N_n$, Halley's method $H_n$, and Chebyshev's method $C_n$, all applied to the polynomial $p_n (z)=z(z^n-1)$ in Table~\ref{tab:methods_comparison}. The information on Newton's and Halley's methods is taken from \cite{HSS2001} and \cite{CGJ2025,halley-liu-etal-2025}, respectively. The formula of $C_n$ is given by Equation~\eqref{eq:Cp2}, whereas
$$N_n(z)=\frac{nz^{n+1}}{(n+1)z^n-1}~~\text{and}~~H_{n}(z)=\frac{n z^{n+1}\left((n+1) z^{n}+(n-1)\right)}{(n+2)(n+1) z^{2 n}+(n+1)(n-4) z^{n}+2}.$$
\begin{table}[htbp]
	\centering
	\begin{tabular}{|p{2.25cm}| p{4.25cm}| p{4.25cm}| p{4.25cm}|}
		% \toprule
		\hline
		\textbf{Description} & \textbf{Newton's method $N_n$} & \textbf{Halley's method $H_n$} & \textbf{Chebyshev's method $C_n$} \\
		% \midrule
		\hline
		\textsl{Degree} & $n+1$ & $2n+1$ & $3n+1$\\
		\hline
		\textsl{Immediate basins corresponding to the non-zero roots of $p_n$} &
		All are unbounded. &
		All are unbounded \cite[Theorem~A]{CGJ2025}.\par   &
		All are unbounded.\par
		\\
		\hline
		\textsl{Immediate basin $\mathcal{A}_0$ of $0$} &  Always unbounded. &  Unbounded if and only if $n \leq 5$ \cite[Theorem~A]{CGJ2025}. & Unbounded if  and only if $n \leq 16$. \\
		% \midrule
		% \midrule
		
		\hline
		\textsl{Complete invariance of   $\mathcal{A}_0$} &
		Always completely invariant. Moreover, $0$ is an exceptional point, hence  $N_n$ is conjugate to a polynomial. &
		Completely invariant if and only if $n \le 5$. But $0$ is not an exceptional point \cite[Proposition~3.6]{CGJ2025}. &
		Never completely invariant, as $C_n: \mathcal{A}_0 \to \mathcal{A}_0 $ has degree $n+1$ or $2n+1$, depending on whether the free critical points are in $\mathcal{A}_0$ or not.\\
		% \midrule
		\hline
		\textsl{Free critical points} &
		No free critical points, giving that $N_n$ is convergent. &
		Free critical points are solutions of $z^n = -\frac{2(n-1)}{(n+1)(n+2)}.$ 
		\par $H_n$ is convergent whenever $n \leq 5$. It is believed but not yet proved that $H_n$ is convergent if $n$ is odd \cite[Theorem~B]{halley-liu-etal-2025}. &
		Free critical points are solutions of $z^n = -\frac{n-1}{(n+1)(2n+1)}.$ 
		\par $C_n$ is convergent whenever $n \leq 16$ or $n$ is odd. For $n \geq 17$, it is not yet proved or disproved. \\
		% \bottomrule
		\hline
		\textsl{Hyperbolic} &
		Always hyperbolic. &   If $H_n$ is convergent, or has no parabolic domain, then $H_n$ is a hyperbolic map. &
		As the poles of $C_n$ are critical points of $C_n$, $C_n$ is never a hyperbolic map. \\
		% \bottomrule
		\hline
	\end{tabular}
	% \vspace{10pt} % Adds a small, clean gap between the table bottom rule and the caption
	\caption{Comparison of Newton's, Halley's, and Chebyshev's methods when applied to $z(z^n -1)$.}
	\label{tab:methods_comparison}
\end{table}

	\section*{Declarations}
	\subsection*{Acknowledgements}  Pooja Phogat is funded by a Senior Research Fellowship (Grant number: 09/1059(0031)/2020-EMR-I) provided by the Council of Scientific and Industrial Research, Govt. of India.  
	
	\subsection*{Conflict of Interests}
	The authors declare that they have no conflict of interest, regarding the publication of this paper.
	
	\subsection*{Data Availability Statement}
	The authors declare that this research is purely theoretical and is not associated with any data.
	%%%%%%%%%%%%%%%%%%%%%%%%%%%%%%%
	%\newpage
	\FloatBarrier

\end{document}